\documentclass{article}
\usepackage[margin=1in]{geometry}
\usepackage{graphicx} 
\usepackage{xcolor}
\usepackage{xparse}
\usepackage{tikz}
\usepackage{booktabs}

\usepackage[round]{natbib}
\usepackage{doi}

\usepackage{algorithm}
\usepackage{algpseudocode}

\usepackage{amsthm}
\usepackage{amsmath}
\usepackage{amsfonts}
\usepackage{mathtools}
\usepackage{bm}
\usepackage{dsfont}
\usepackage{hyperref}
\usepackage[nameinlink]{cleveref}
\usepackage{listings}

\newtheorem{theorem}{Theorem}
\newtheorem{lemma}[theorem]{Lemma}
\newtheorem{corollary}[theorem]{Corollary}

\crefname{lemma}{lemma}{lemmas}
\Crefname{lemma}{Lemma}{Lemmas}
\crefname{theorem}{Theorem}{Theorems}
\Crefname{theorem}{Theorem}{Theorems}
\crefname{corollary}{corollary}{corollaries}
\Crefname{corollary}{Corollary}{Corollaries}
\crefname{equation}{}{}
\Crefname{equation}{}{}

\newcommand{\setI}{\mathcal{I}}
\newcommand{\setM}{\mathcal{M}}
\newcommand{\setP}{\mathcal{P}}

\newcommand\norm[2][{}]{\ensuremath{\|#2\|}_{#1}}

\newcommand\inv[2][1]{\left( #2 \right)^{-#1}}

\newcommand\Biginv[2][1]{\paren[\Big]{#2}^{-#1}}
\newcommand\expect[2][{}]{\ensuremath{\E_{#1} \bracket{#2}}}

\newcommand{\vect}[1]{\mathbf{#1}}

\newcommand{\va}{\vect{a}}  
\newcommand{\vb}{\vect{b}}
\newcommand{\vc}{\vect{c}}  
\newcommand{\vd}{\vect{d}}
\newcommand{\ve}{\vect{e}}

\newcommand{\vu}{\vect{u}}  
\newcommand{\vv}{\vect{v}}  
\newcommand{\vw}{\vect{w}}
\newcommand{\vx}{\vect{x}}  
\newcommand{\vy}{\vect{y}}  
\newcommand{\vz}{\vect{z}}

\newcommand{\mA}{\vect{A}}

\newcommand{\mD}{\vect{D}}

\newcommand{\mG}{\vect{G}}
\newcommand{\mH}{\vect{H}}
\newcommand{\mI}{\vect{I}}
\newcommand{\mJ}{\vect{J}}
\newcommand{\mK}{\vect{K}}

\newcommand{\mM}{\vect{M}}
\newcommand{\mN}{\vect{N}}

\newcommand{\mP}{\vect{P}}

\newcommand{\mT}{\vect{T}}

\newcommand{\mW}{\vect{W}}
\newcommand{\mX}{\vect{X}}

\newcommand{\mZ}{\vect{Z}}

\newcommand{\setS}{\mathcal{S}}
\newcommand{\setSbar}{{\overline{\mathcal{S}}}}
\newcommand{\setE}{\mathcal{E}}

\newcommand{\grvect}[1]{{\bm{#1}}}

\newcommand{\vbeta}{\grvect{\beta}}

\newcommand{\vtheta}{\grvect{\theta}} 
 
\newcommand{\vlambda}{\grvect{\lambda}} 
\newcommand{\vmu}{\grvect{\mu}} 
\newcommand{\vnu}{\grvect{\nu}}

\newcommand{\mTheta}{\grvect{\Theta}}

\newcommand{\mPi}{\grvect{\Pi}}
\newcommand{\mSigma}{\grvect{\Sigma}}

\DeclareMathOperator*{\argmin}{arg\,min}

\newcommand{\diag}[2][{}]{\mathrm{diag}_{#1}( #2)}
\newcommand{\set}[2][{}]{#1\{ #2 #1\}}
\newcommand{\paren}[2][{}]{#1( #2 #1)}

\newcommand{\bracket}[2][{}]{#1[ #2 #1]}

\def\vone{{\bf 1}}
\def\vzero{{\bf 0}}
\def\E{{\mathbb E}}

\def\ind{{\mathds 1}}
\def\defeq{:=}
\def\rdefeq{=:}
\newcommand{\asconv}{\xrightarrow{\mathrm{a.s.}}}

\newcommand{\dconv}{\xrightarrow{\mathrm{d}}}
\def\vbetahat{{\widehat{\vbeta}}}

\newcommand{\complexset}{\mathbb C}
\newcommand{\reals}{\mathbb{R}}
\newcommand{\nats}{\mathbb{N}}

\newcommand{\transp}{\top}

\newcommand{\cv}{\mathrm{CV}}
\newcommand{\iid}{\mathrm{i.i.d.}}
\newcommand{\normal}{\mathcal{N}}

\newcommand{\tr}{\mathrm{tr}}

\NewDocumentCommand{\intdx}{s O{} O{} m O{x}}{
  \IfBooleanTF {#1}{%
      \int_{#2}^{#3} {#4} \, {#5}%
  }{%
      \int_{#2}^{#3} {#4} \,\mathrm{d}{#5}%
  }%
}

\definecolor{codegreen}{rgb}{0,0.6,0}
\definecolor{codegray}{rgb}{0.5,0.5,0.5}
\definecolor{codepurple}{rgb}{0.58,0,0.82}
\definecolor{backcolor}{HTML}{f0f3f5}

\lstdefinestyle{mystyle}{
    backgroundcolor=\color{backcolor},   
    commentstyle=\color{codegreen},
    keywordstyle=\color{magenta},
    numberstyle=\tiny\color{codegray},
    stringstyle=\color{codepurple},
    basicstyle=\ttfamily\footnotesize,
    breakatwhitespace=false,         
    breaklines=true,                 
    captionpos=b,                    
    keepspaces=true,                 
    numbers=left,                    
    numbersep=5pt,                  
    showspaces=false,                
    showstringspaces=false,
    showtabs=false,                  
    tabsize=2
}

\usepackage{authblk}
\usepackage{comment}

\title{RandALO: Out-of-sample risk estimation in no time flat}
\date{}

\author[1]{Parth T. Nobel}
\author[2]{Daniel LeJeune}
\author[2]{Emmanuel J.\ Cand\`{e}s}
\affil[1]{Department of Electrical Engineering, Stanford University}
\affil[2]{Department of Statistics, Stanford University}
\affil[ ]{\texttt{ptnobel@stanford.edu, daniel@dlej.net, candes@stanford.edu}}

\begin{document}

\maketitle

\begin{abstract}
Estimating out-of-sample risk for models trained on large high-dimensional datasets is an expensive but essential part of the machine learning process, enabling practitioners to optimally tune hyperparameters. Cross-validation (CV) serves as the de facto standard for risk estimation but poorly trades off high bias ($K$-fold CV) for computational cost (leave-one-out CV). We propose a randomized approximate leave-one-out (RandALO) risk estimator that is not only a consistent estimator of risk in high dimensions but also less computationally expensive than $K$-fold CV. We support our claims with extensive simulations on synthetic and real data and provide a user-friendly Python package implementing RandALO available on PyPI as \texttt{randalo} and at \url{https://github.com/cvxgrp/randalo}.
\end{abstract}

\section{Introduction}

Training machine learning models is an often expensive process, especially in large data settings. Not only is there significant cost in the fitting of individual models, but even more importantly, the best model must be chosen from a set of candidates parameterized by a set of ``hyperparameters'' indexing the models, and each of these models must be fitted and evaluated in order to make the optimal selection. As a result, model selection, also called hyperparameter tuning, tends to be the most computationally expensive part of the machine learning pipeline.

In order to evaluate models, we typically need to set aside unseen ``holdout'' data to estimate the risk of the model on new samples from the training distribution. When we have an abundance of training samples, such as in the millions or billions, we can afford to set aside a modest holdout set of tens of thousands of examples without compromising model performance. We can then simply evaluate the fitted model on the holdout set and obtain a high precision estimate of model risk, with the only major cost being the fitting of the model on the training data.

In even moderately large data regimes, however, when we have at most tens of thousands of possibly high-dimensional training samples, it is often not possible to set aside a sufficiently large holdout set without sacrificing the quality of our model fit. In these settings, the time-trusted technique for model evaluation is $K$-fold cross-validation (CV): the data is partitioned into $K$ roughly equal subsets, and each subset is used as a holdout set while the model is trained on the remaining data, and finally the model risks across each of the $K$ folds are averaged. In this way, we get the advantage of evaluating our model on a set of data the same size as the training data.

The downsides of CV are two-fold. Firstly, for each of the $K$ folds, a new model must be fit, increasing the computational cost of evaluating risk, and thus model selection, by roughly a factor of $K$.\footnote{The cost of training individual models on the $(K-1)/K$ fraction of the data is generally a bit less than the cost of training a model on the full dataset, so the total cost is a little less than $K$ times.} Secondly, and perhaps more alarmingly, $K$-fold CV provides an unbiased estimate only for the risk of a model trained on $n (K-1)/K$ data points, which can be quite different from the risk of a model trained on $n$ points in high dimensions \citep{donoho2011phase}. This bias only vanishes as $K$ approaches the number of training samples, at which point it becomes known as leave-one-out CV, at the expense of a tremendous computational cost.

\begin{figure}[t]
    \centering
    \includegraphics[width=\linewidth]{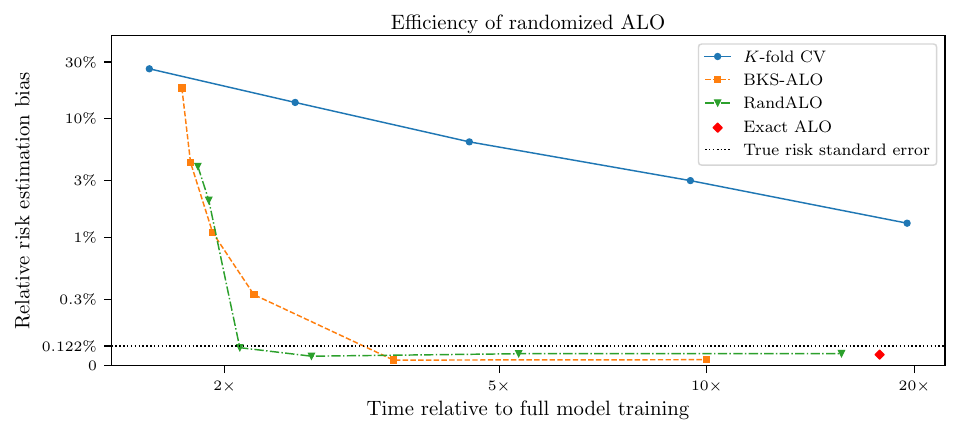}
    \caption{$K$-fold cross-validation (CV, solid blue, circles) provides a poor trade-off between risk estimation error and computational time on a high-dimensional lasso problem. Meanwhile, BKS-ALO (dashed orange, squares), a simplified version of our method, dominates CV in estimation bias and computational cost. Our fully debiased procedure RandALO (dash--dot green, triangles) goes further and reduces bias by an order of magnitude for the same computational cost, and both methods reach the same bias as exact ALO (red diamond) in a fraction of the time. Lines denote mean risk estimate bias and time over 100 trials. We report the relative risk estimation bias is computed as $|\hat{R} - R| / R$ for a particular mean risk estimate $\hat{R}$, where the true risk $R$ is estimated as the sample mean of the conditional risks given the training data. The $y$-axis is logarithmic above the true conditional risk standard error of $0.122\%$ (dotted, black) and linear below.}
    \label{fig:cv-tradeoff}
\end{figure}

In this work, we propose a randomzied risk estimation procedure (RandALO) that addresses both of these issues. \Cref{fig:cv-tradeoff} compares bias and real-world wall-clock time for $K$-fold CV to our method RandALO on a high-dimensional lasso problem (experimental details in  \Cref{sec:num:lasso}). Regardless of the choice of $K \in \set{2, 3, 5, 10, 20}$, we can implement our method with some choice of $m \in \set{10, 30, 100, 300, 1000, 3000}$ Jacobian--vector products and achieve lower bias and lower computational cost. 
RandALO provides very high quality risk estimates with only $0.1\%$ bias in around $2 \times$ the time of training a model, while we would have to use $K=20$ and nearly $20 \times$ the training time to achieve merely $1\%$ bias with $K$-fold CV.

Our method is based on the approximate leave-one-out (ALO) technique of \citet{rad2020alo}, which approximates leave-one-out CV using a single step of Newton's method for each training point. This technique has been shown to enjoy the same consistency properties as leave-one-out CV for large high-dimensional datasets, thus being more accurate than $K$-fold CV.
However, a barrier to applying ALO is its poor scaling with dataset size.
To address this,
we employ randomized numerical linear algebra techniques to reduce the computation down to the cost of solving a constant number of quadratic programs involving the training data, which is enough to be computational advantageous against even the low cost of $5$-fold CV with highly optimized solvers for methods such as the lasso.
For non-standard and less optimized solvers, the computational advantage is even more dramatic.

We have also created a Python package, available on PyPI as \texttt{randalo} and at \url{https://github.com/cvxgrp/randalo}, that makes applying RandALO as simple as cross-validation. 
Users can use their solver of choice to first fit the model on all of the training data, and then use our package to estimate its risk.
For example, for an ordinary scikit-learn~\citep{scikit-learn} \texttt{Lasso} model, we can obtain a RandALO risk estimate with a single additional line of code:

\begin{lstlisting}[language=python]
X, y = ... # collect training data
lasso = Lasso(1.0).fit(X, y) # fit the model
risk_estimate = RandALO.from_sklearn(lasso, X, y).evaluate(MSELoss()) # estimate risk
\end{lstlisting}

\paragraph{Contributions.}
Concretely, our contributions are as follows:
\begin{enumerate}
    \item We develop a randomized method RandALO (\Cref{alg:alo-randomized}) for efficiently and accurately computing ALO given access to a Jacobian--vector product oracle for the fully trained model. 
    \item We prove the asymptotic normality and decorrelation of randomized diagonal estimation for Jacobians of generalized ridge models with high-dimensional elliptical sub-exponential data (\Cref{thm:bks-clt}).
    \item We show that Jacobian--vector products for linear models with non-smooth regularizers can be computed efficiently via appropriate quadratic programs (\Cref{thm:derivative-as-qp}).
    \item We provide extensive experiments in \Cref{sec:num} demonstrating the advantage of randomized ALO over $K$-fold CV in terms of both risk estimation and computation across a wide variety of linear models and high-dimensional synthetic and real-world datasets.
    \item We provide a Python software package that enables the easy application of RandALO to real-world machine learning workflows.
\end{enumerate}



\subsection{Related work}

Risk estimation is an important aspect of model selection and has a wide literature and long history in machine learning: we refer the reader to \citet{friedman_hastie_tibshirani_2009}, Chapter 7; \citet{arlot_celisse_2010}; and \citet{zhang_yang_2015} for an overview of common techniques, particularly emphasizing cross-validation (CV). Generalized cross-validation (GCV) \citep{craven_wahba_1979} is an approximation to leave-one-out (LOO) CV that can be efficiently implemented using randomized methods \citep{hutchinson1989trace} and has recently been shown to be consistent for linear in high dimensions under certain random matrix assumptions~\citep{patil2021uniform,bellec2020out} and in sketched models~\citep{patil2024asymptotically}.

Approximate leave-one-out (ALO) generalizes GCV by performing individual Newton steps toward the LOO objective for each training point rather than making a single uniform correction, coinciding with exact LOO for ridge regression. ALO was proposed for logistic regression  \citep{pregibon1981logistic} and kernel logistic regression \citep{Cawley2008} via iteratively reweighted least squares. More recently, \citet{rad2020alo} showed consistency of ALO under appropriate random data assumptions in high dimensions for arbitrary losses and regularizers \citep{xu2021consistent}, including non-smooth penalties such as $\ell_1$ \citep{auddy2024approximate}.
Related and complementary to our work,
\citet{stephenson2020approximate} propose to exploit the sparsity in the linear model to reduce computation in ALO.
In another direction, \citet{luo2023iterative} share our aim in making ALO more useful by providing an extension to models obtained from iterative solvers that is consistent along the entire optimization path. Techniques based on approximate leave-one-out have also been among the most successful methods for data attribution in deep learning \citep{park2023trak}.

Randomized numerical methods and risk estimation have been paired for decades, with the randomized trace estimator of \citet{hutchinson1989trace} originally proposed for estimating the degrees of freedom quantity in GCV used for correcting residuals. \citet{bekas2007diagonal} extended Hutchinson's approach to a method for estimating the diagonal elements of a matrix, which we base our method upon. \citet{baston2022stochastic} combined this estimator with ideas from Hutch++ \citep{meyer2021hutch} to create an improved estimator called Diag++ when the computational budget can be directed towards a prominent low rank component, and \citet{epperly2024xtrace} improve this estimator by enforcing exchangeability in their method XDiag. We found both Diag++ and XDiag to be less accurate than the procedure of \citet{bekas2007diagonal} when the number of matrix--vector products is restricted to be much below the effective rank of the matrix and so did not incorporate them into our method.

\section{Approximate leave-one-out for linear models}
\label{sec:alo}

We consider the class of linear models obtained by empirical risk minimization, having the form
\begin{align}
    \widehat{\vbeta} = \argmin_\vbeta \sum_{i=1}^n \ell(y_i, \vx_i^\transp \vbeta) + r(\vbeta),
    \label{eq:erm}
\end{align}
where $((\vx_i, y_i))_{i=1}^n \in (\reals^p \times \reals)^n$ is an i.i.d.\ training data set, $\ell \colon \reals \times \reals \to \reals$ is a twice differentiable loss function, and $r \colon \reals^p \to \reals$ is a possibly non-differentiable regularizer. In model selection, we aim to optimize the \emph{risk} of the model for some risk function $\phi \colon \reals \times \reals \to \reals$:
\begin{align}
    R \defeq \expect{\phi(y, \vx^\transp \widehat\vbeta)}.
\end{align}
Here the expectation is taken over an independent sample $(\vx, y)$ drawn from the same distribution as the training data, as well as over the training data used to fit the model $\widehat{\vbeta}$.\footnote{While in principle one would prefer the conditional risk given $\widehat{\vbeta}$, cross-validation is only able to estimate marginal risk \citep{bates2023cv}.}
Given a partition $\pi$ of $[n]$,
the cross-validation (CV) family of risk estimators have the following form:
\begin{align}
    \hat{R}_\cv \defeq \frac{1}{n} \sum_{\setP \in \pi} \sum_{i \in \setP} \phi(y_i, \vx_i^\transp \widehat{\vbeta}_{-\setP})
    \quad \text{for} \quad
    \widehat{\vbeta}_{-\setP} \defeq \argmin_\vbeta \sum_{i \notin \setP} \ell(y_i, \vx_i^\transp \vbeta) + r(\vbeta).
    \label{eq:cv}
\end{align}
The popular $K$-fold CV consists of a partition of $K$ subsets approximately equal in size, while leave-one-out (LOO) CV uses the partition of singletons $\pi = \set{\set{1}, \set{2}, \ldots, \set{n}}$. Choosing the cross-validation partition is an accuracy--computation trade-off, as using a few large subsets (as in $K$-fold CV) results in relatively quick but low-quality biased risk estimates, while using many small subsets (as in LOO CV) results in high-quality but computationally intensive risk estimates.

Leave-one-out risk estimation na\"ively has a cost of essentially $n$ times the cost of training the model, which becomes prohibitive for moderately large $n$. However, in certain settings, there are ``shortcut'' formulas that enable the efficient computation of the LOO predictions starting from the model trained on all of the data. 
Notably, in the setting of ridge regression where $\ell(y, z) = (y - z)^2$ and $r(\vbeta) = \lambda \norm[2]{\vbeta}^2$,
we have the following exact form of the LOO prediction $\hat{y}_{-i} \defeq \vx_i^\transp \widehat{\vbeta}_{-\set{i}}$ in terms of the full prediction $\hat{y}_i \defeq \vx_i^\transp \widehat{\vbeta}$ and the $n \times p$ data matrix $\mX$ formed by stacking the data $\vx_i$ as rows:
\begin{align}
    \hat{y}_{-i} = \frac{\hat{y}_i - J_{ii} y_i}{1 - J_{ii}}
    \quad\text{where}\quad
    \mJ = \mX \inv{\mX^\transp \mX + \lambda \mI} \mX^\transp.
\end{align}
The matrix $\mJ$ is often called the ``hat'' matrix or linear smoothing matrix, and it is also the Jacobian matrix $\mJ = \partial \widehat{\vy} / \partial \vy$ of predictions of the model $\widehat{\vy} = (\hat{y}_i)_{i=1}^n$ trained on the full data with respect to the training labels $\vy = (y_i)_{i=1}^n$, since $\widehat{\vy} = \mJ \vy$. 
Computing $\mJ$ and extracting $J_{ii}$ has the same computational complexity as computing the exact full ridge solution $\widehat{\vbeta}$, such as when using a Cholesky decomposition of $\mX^\transp \mX + \lambda \mI$ which can be reused to compute $\mJ$, and so LOO CV can be performed at minimal additional cost.


Outside of ridge regression, LOO can be approximated for each data point by a single step of Netwon's method towards minimizing the CV objective in the right-hand side of \cref{eq:cv} starting from the full solution $\widehat{\vbeta}$. 
This idea was proposed for logistic regression as early as \cite{pregibon1981logistic}, but was only recently proven to provide consistent risk estimation in high dimensions  \citep{rad2020alo}. For general twice differentiable losses and regularizers (note: ALO can also be applied with non-smooth regularizers, as we describe in \Cref{sec:jvp}), the approximate leave-one-out (ALO) prediction is given by
\begin{align}
    \tilde{y}_{i} \defeq \hat{y}_i +  \frac{\ell'(y_i, \hat{y}_i) \tilde{J}_{ii}}{\ell''(y_i, \hat{y}_i) (1 - \tilde{J}_{ii})},
    \quad \text{where} \quad
    \widetilde{\mJ} = \mX (\mX^\transp \mH_\ell \mX + \nabla^2 r(\widehat{\vbeta}))^{-1} \mX^\top \mH_\ell.
    \label{eq:alo}
\end{align}
Here $\ell'(y, z) = \partial \ell(y, z) / \partial z$ and $\ell''(y, z) = \partial^2 \ell(y, z) / \partial z^2$, while $\mH_\ell = \diag{(\ell''(y_i, \hat{y}_i)_{i=1}^n}$. The matrix $\widetilde{\mJ}$ is closely related to the Jacobian $\mJ = \partial \widehat{\vy}/\partial \vy$ via the scaling transformation $\tilde{J}_{ij} = -J_{ij} \ell''(y_j, \hat{y}_j) / (\partial \ell'(y_j, \hat{y}_j) / \partial y_j)$, and so with some abuse of nomenclature we also refer to $\widetilde{\mJ}$ as the Jacobian. Note that ALO coincides exactly with LOO in the case of ridge regression. We also note that by formulating ALO in terms of the Jacobian, the corrected predictions $\tilde{y}_{i}$ can be computed without dependence on a specific parameterization of the prediction function. While we consider linear models for the remainder of this work, we describe in \Cref{sec:generic-alo} how ALO can be derived and applied for arbitrary nonlinear prediction functions.

Although LOO and ALO have similar cost to obtaining models via direct methods such as matrix inversion, 
predictors in machine learning 
are often found by iterative algorithms that return a high-quality approximate solution very quickly. In ridge regression for example, a solution can be found in time $O(\sqrt{\kappa} np)$ where $\kappa$ is the condition number of $\mX^\transp \mX + \lambda \mI$ (see, e.g., \citealp[\S11.3]{golub2013matrix}). This can be significantly less time than the $O(p^3)$ time needed to perform the inversion in the computation of $\mJ$ using direct numerical techniques. In general, this means that it can be significantly faster to obtain a risk estimate via $K$-fold CV which does not require computing the diagonals of $\widetilde{\mJ}$, making exact ALO unattractive compared to $K$-fold CV for large-scale data. However, by leveraging randomized methods, we can exploit the same iterative algorithms to approximate ALO in time comparable to or better than the solving for the original full predictor, yielding a risk estimate that often outperforms $K$-fold CV in both accuracy and computational cost.

\section{Randomized approximate leave-one-out}

The computational bottleneck of ALO lies in extracting the diagonals of the Jacobian $\widetilde{\mJ}$. Doing this exactly requires first realizing the full matrix $\widetilde{\mJ}$, which is prohibitive in large data settings. This motivates our use of randomized techniques for estimating the diagonal elements of $\widetilde{\mJ}$ using the extension of Hutchinson's trace estimator
proposed by \citet{bekas2007diagonal}. We refer to this method as ``BKS'' after the names of the authors. This randomized diagonal estimate requires only $m$ Jacobian--vector products:
\begin{align}
    \vmu = \frac{1}{m} \sum_{k=1}^m (\widetilde{\mJ} \vw_k) \odot \vw_k,
\end{align}
where $\vw_k \in \reals^n$ are i.i.d.\ Rademacher random vectors, taking values $\set{\pm 1}$ with equal probability, and $\odot$ denotes the element-wise product. It is straightforward to see that $\expect{\vmu} = \diag{\widetilde{\mJ}}$ and that $\vmu$ converges to $\diag{\widetilde{\mJ}}$ almost surely as $m \to \infty$ by the strong law of large numbers. However, we need to keep $m$ small in order to minimize computation; using $m = n$ for example would have similar computational cost to evaluating $\widetilde{\mJ}$ in full.

Fortunately, for large scale machine learning problems, the number of Jacobian--vector products needed to reach a desired level of accuracy generally does not increase with the scale of the problem.
We capture this formally with the following theorem under elliptical random matrix assumptions on the data, specializing to the case of generalized ridge regression with the regularizer $r(\vbeta) = \tfrac{1}{2} \vbeta^\transp \mG \vbeta$. Note that this matches the general form of \cref{eq:alo}, aside from possible dependence of $\mG$ on $\vbeta$. We prove this result for sub-exponential random data with arbitrary covariance structure and thus can expect the takeaways to be quite general.

\begin{theorem}
    \label{thm:bks-clt}
    Let $\widetilde{\mJ} = \mX \inv{\mX^\transp \mX + \mG} \mX^\transp$ for $\mX = \mT^{1/2} \mZ \mSigma^{1/2}$, where $\mT \in \reals^{n \times n}$ is a diagonal matrix with positive diagonal elements $t_i$ in a finite interval $\setI$ separated from $0$, $\mZ \in \reals^{n \times p}$ are i.i.d.\ zero-mean and unit-variance $\alpha$-sub-exponential variables for some $0 < \alpha \leq 1$, and $\mSigma, \mG \in \reals^{p \times p}$ are positive semidefinite matrices such that $n\mG^{-1/2}\mSigma \mG^{-1/2}$ has eigenvalues in $\setI$. Then for any $m \in \nats$ and $i \neq j \in [n]$, we have the following convergence almost surely in the limit as $n, p \to \infty$ such that $0 < \liminf n/p \leq \limsup n/p < \infty$, 
    \begin{align}
        \frac{\mu_i - \tilde{J}_{ii}}{\frac{\sqrt{t_i \nu}}{(1 + t_i \eta)}} \,\Big|\, \mX 
        \dconv \normal(0, \tfrac{1}{m}) 
        \quad\text{and}\quad
        \E \bracket{\paren{\mu_i - \tilde{J}_{ii}}
        \paren{\mu_j - \tilde{J}_{jj}} \,|\, \mX
        } \to 0,
        \label{eq:thm:bks-clt}
    \end{align}
    where $\eta = \tr [\mSigma \inv{\mX^\transp \mX + \mG}]$ and $\nu = \tr [\mSigma \inv{\mX^\transp \mX + \mG} \mX^\transp \mX \inv{\mX^\transp \mX + \mG}]$. 
\end{theorem}

\begin{proof}[Proof sketch]
As discussed above, the BKS estimator is unbiased and so $\expect{\mu_i} = \expect{\tilde{J}_{ii}}$. In the case $m = 1$, the remainder has the form $\mu_i - \tilde{J}_{ii} = \sum_{j \neq i} w_i \vx_i^\transp (\mX^\transp \mX + \mG)^{-1} \vx_j w_j$. We then apply Lyapunov's central limit theorem over the randomness in $\vw$ with an appropriate argument that the terms of the sum are not too sparse. Application of the Woodbury identity allows us to extract $\vx_i$ from inside the inverse and determine the influence of $t_i$, and then using random matrix theory we can argue that $\vx_i$ can be reintroduced to obtain the limiting values of $\eta$ and $\nu$. To show uncorrelatedness, we observe that $\E \bracket{\paren{\mu_i - \tilde{J}_{ii}}
\paren{\mu_j - \tilde{J}_{jj}} \,|\, \mX
} = (\vx_i^\transp (\mX^\transp \mX + \mG)^{-1} \vx_j)^2$, which asymptotically vanishes due to the independence of $\vx_i$ and $\vx_j$.
We give the proof details along with the definition of $\alpha$-sub-exponential variables in \Cref{sec:thm:bks-clt:proof}.
\end{proof}

Although the statement of \Cref{thm:bks-clt} is asymptotic in $n$ and $p$, the distributional convergence occurs very quickly. We illustrate this on a small problem with $n = 200$ and $p = 150$ in \Cref{fig:bks_viz}. We generate a single fixed dataset with $Z_{ij}$ taking values $\set{\pm 1}$ with equal probability, $\sqrt{t}_i \in \set{1, 2, 3, 4}$ with equal proportion, $\mSigma = \diag{(\sigma_j^2)_{j=1}^p}$ with $\sigma_j \in \set{1, 2}$ with equal proportion, and $\mG = n \mI$. Then over 1000 random trials of drawing $\vw_k$ and computing $m=10$ Jacobian--vector products, we examine the empirical distributions of the results. First, in the left-most plot, we compare the histogram of the $\mu_i$'s ($n \times 1000 = {}$200,000 values total) to the asymptotic density
\begin{align}
    dF(x) = 
    \frac{1}{n} \sum_{i=1}^n \frac{1 + t_i \eta}{\sqrt{2 \pi t_i \nu}} \exp \set[\Bigg]{-\frac{(x - \tilde{J}_{ii})^2 (1 + t_i \eta)^2}{2 t_i \nu}} dx,
\end{align}
which is the convolution of the Gaussian error distributions from \Cref{thm:bks-clt} with the point masses at each $\tilde{J}_{ii}$, revealing a near-exact match. In the middle plot, we compare the histogram of $z$-scores---that is, left-most expression in \cref{eq:thm:bks-clt}---another 200,000 values, with the standard normal density, but this time for only $m=1$ to show that Gaussianity arises without any averaging. In the right-most plot, we demonstrate the uncorrelatedness of the errors by taking the $z$-scores of a single Jacobian--vector product and showing that successive pairs $z_i$ and $z_{i+1}$ (199 pairs total) are uncorrelated, as predicted, meaning that we can treat the errors as independent in downstream analysis.

The key takeaway for developing an efficient method for computing ALO is that since the $\eta$ and $\nu$ are $O(1)$ even as $n, p \to \infty$, the variance of the $\mu_i$'s are $O(1/m)$ regardless of problem size. Furthermore, since the noise is asymptotically uncorrelated across $i$, it is easy to understand the effect on estimation error downstream. On the other hand, however, the quantities $\eta$ and $\nu$ also do not tend to vanish with problem size, so if we want to minimize the number of Jacobian--vector products, we must be able to handle this non-vanishing noise effectively.

\begin{figure}[t]
    \centering
    \includegraphics[width=\linewidth]{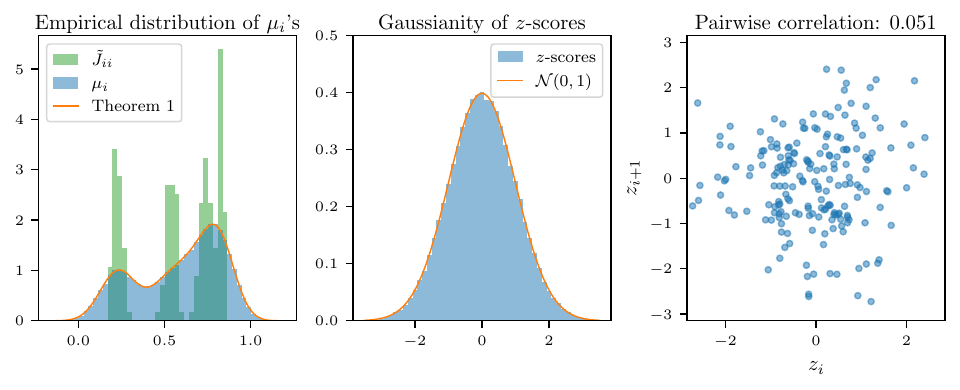}
    \caption{\Cref{thm:bks-clt} provides a very accurate characterization of randomized diagonal estimation even for a fairly small problem with $n = 200$ and $p = 150$. \textbf{Left:} The empirical distribution of $\mu_i$ for $m=10$ over 1000 trials (with randomness only over the vectors $\vw_k$) exactly matches the mixture of Gaussians centered at each $\tilde{J}_{ii}$ predicted by \Cref{thm:bks-clt}. 
    \textbf{Middle:} Taking the $z$-scores of the individual Jacobian--vector products $(\widetilde{\mJ} \vw_k) \odot \vw_k$ from the same experiment, the empirical distribution is well described by the standard normal. 
    \textbf{Right:} Looking at the $z$-scores for a single Jacobian--vector product, the pairs of successive elements of the resulting vector are uncorrelated as predicted by the asymptotics.}
    \label{fig:bks_viz}
\end{figure}

\subsection{Dealing with noise: Inversion sensitivity}

When applying ALO, each corrected prediction $\tilde{y}_i$ is an affine function of $\tilde{J}_{ii} / (1 - \tilde{J}_{ii})$. Because of the division by $1 - \tilde{J}_{ii}$, this means that ALO is extremely sensitive to noise in the estimation of $\tilde{J}_{ii}$. Unfortunately, by \Cref{thm:bks-clt}, we know that for any $m$, there is always some nonzero probability that $\mu_i$ will be near or greater than 1, which means that na\"ively plugging in the BKS diagonal estimate would sometimes result in extremely incorrect predictions $\tilde{y}_i$.

It is straightforward to check that the values $\tilde{J}_{ii}$ must lie in $[0, 1]$, so one possible solution would be to clip the values of $\mu_i$ if they are larger than $1 - \varepsilon$ to avoid near division by 0 or an incorrect sign. However, it is difficult to choose $\varepsilon$ robustly: for example, as ridge regression approaches interpolating least squares in high dimensions, the values of $J_{ii}$ approach 1, so a non-adaptive value of $\varepsilon$ would give an inconsistent result. A more sophisticated approach would be to perform minimum mean squared error (MMSE) estimation of $\tilde{J}_{ii}$ given the individual diagonal estimates $(\widetilde{\mJ} \vw_k) \odot \vw_k$ and an appropriate prior on $\tilde{J}_{ii}$ supported on $[0, 1]$. Since these diagonal estimates are uncorrelated Gaussian variables by \Cref{thm:bks-clt}, the sufficient statistics for this estimation problem are the sample means $\mu_i$ and the sample variances
\begin{align}
    \sigma_i^2 = \frac{1}{m - 1} \sum_{k=1}^m ([(\widetilde{\mJ} \vw_k) \odot \vw_k]_i - \mu_i)^2.
\end{align}
If we assume that $\sigma_i^2$'s are sufficiently well estimated to plug in in place of the true variances, we can compute the posterior distribution of $\tilde{J}_{ii}$ given $\mu_i$ and $\sigma_i^2$ for a uniform prior on $[0, 1]$:
\begin{align}
    p(\tilde{J}_{ii} | \mu_i, \sigma_i^2) \propto p(\mu_i | \tilde{J}_{ii}, \sigma_i^2) \ind_{[0, 1]}(\tilde{J}_{ii})
    \propto 
    \exp\set[\Big]{-\frac{m (\mu_i - \tilde{J}_{ii})^2}{2 \sigma_i^2}} \ind_{[0, 1]}(\tilde{J}_{ii}).
\end{align}
Thus, the MMSE estimator is the conditional mean of this distribution, which is the mean of a truncated normal distribution with location and scale parameters $\mu_i$ and $\sigma_i / \sqrt{m}$ on the interval $[0, 1]$. Numerical functions for evaluating this mean are commonly available in standard scientific computing packages, making this estimator computationally efficient and trivial to implement. We demonstrate the effectiveness of this approach in \Cref{fig:truncnorm-viz}, and we refer to the method in which we plug in the MMSE estimates of $\tilde{J}_{ii}$ into \cref{eq:alo} and use the resulting $\tilde{y}_i$ to evaluate risk as BKS-ALO.

This simple MMSE estimation approach enables us to gracefully handle noise that would otherwise need to be handled downstream in the algorithm using ad hoc outlier detection. Furthermore, it still has asymptotic normality in $m$, matching the distribution of $\mu_i$, and so makes it straightforward to reason about the remaining bias introduced by noisy estimation of diagonals in ALO.

\begin{figure}[t]
    \centering
    \includegraphics[width=\textwidth]{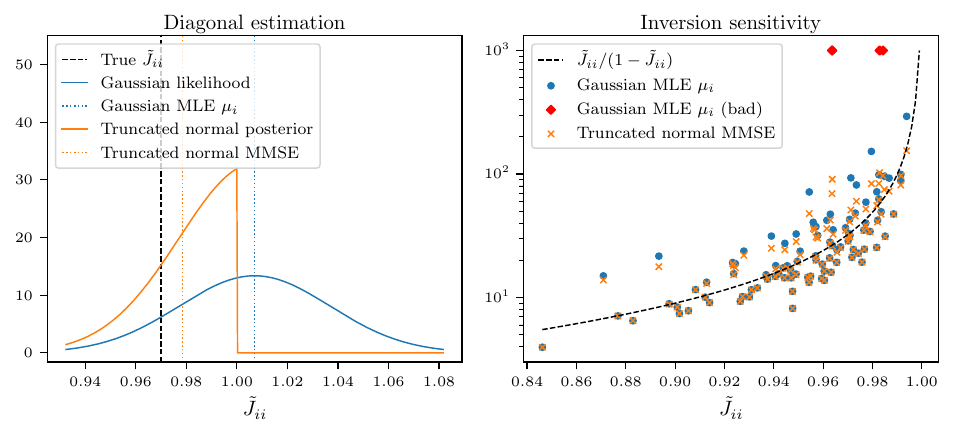}
    \caption{
    \textbf{Left:} Minimum mean squared error (MMSE) estimation using a uniform prior on $\tilde{J}_{ii}$ (orange) provides a much better estimate than the na\"ive maximum likelihood estimate (MLE) $\mu_i$ (blue), and is meaningful even when $\mu_i \notin [0, 1]$.
    \textbf{Right:} We plug in our diagonal estimates into the formula $\tilde{J}_{ii}/(1 - \tilde{J}_{ii})$ for a ridge regression problem with $n = p = 100$, $\lambda = 0.1$, and $\vx_i = t_i \vz_i$ for $t_i \sim \mathrm{Uniform}[\tfrac{1}{2}, 1]$ and $\vz_i \sim \normal(\vzero, \mI)$. Direct application of $\mu_i$ for $m = 50$ provides poor estimates when $\tilde{J}_{ii}$ is close to 1 (blue  circles), and sometimes yields nonsense results when $\mu_i > 1$ (red diamonds) which are poorly addressed by clipping. Meanwhile, the truncated normal MMSE strategy (orange $\times$'s) controls the effect of noise on inversion.
    }
    \label{fig:truncnorm-viz}
\end{figure}

\subsection{Dealing with noise: Risk inflation debiasing}

Since the variances of $\mu_i$ decay only at a rate of $1 / m$, the effects of estimation noise will appear in our risk estimate for any finite $m$. Unfortunately, as demonstrated in \Cref{fig:truncnorm-viz} (right), the noise can be substantial even for moderate numbers of Jacobian--vector products such as $m = 50$. The effect of this noise is typically a bias towards an inflated estimate of risk, larger than the true risk. To see why this is generally the case, consider a sufficiently large $m$ such that Gaussianity is preserved through the mapping $\mu \mapsto \mu / (1 - \mu)$ and thus for some $s_i > 0$,
\begin{align}
    \frac{\ell'(y_i, \hat{y}_i)}{\ell''(y_i, \hat{y}_i)}\frac{\mu_i}{1 - \mu_i} \sim \normal \paren[\Big]{\frac{\ell'(y_i, \hat{y}_i)}{\ell''(y_i, \hat{y}_i)}\frac{\tilde{J}_{ii}}{1 - \tilde{J}_{ii}}, \frac{s_i^2}{m}}.
\end{align}
Let $\bar{y}_i$ be the ALO-corrected prediction evaluated at $\mu_i$ instead of $\tilde{J}_{ii}$, such that $\bar{y}_i = \tilde{y}_i + s_i z_i/\sqrt{m}$ for $z_i \sim \normal(0, 1)$ independent of $y_i$, $\tilde{y}_i$, and $s_i$. Now consider any convex risk function $\phi$. By the law of large numbers and Jensen's inequality, we should expect to see bias for large $n$. Letting $Y$, $\widetilde{Y}$, and $S$ denote random variables matching the empirical distributions of $y_i$, $\tilde{y}_i$, and $s_i$ with an independent $Z \sim \normal(0, 1)$:
\begin{align}
    \frac{1}{n} \sum_{i=1}^n \phi(y_i, \bar{y}_i) \approx \E \bracket[\Big]{\phi \paren[\Big]{Y, \widetilde{Y} + \frac{SZ}{\sqrt{m}}}} \geq \expect{\phi(Y, \widetilde{Y})} \approx \frac{1}{n} \sum_{i=1}^n \phi(y_i, \tilde{y}_i).
\end{align}
The risk estimate may be inflated even if the risk function is not convex, such as for ``zero--one'' error $\phi(y, z) = \ind\set{yz < 0}$, since noisier predictions generally incur higher risk, and the estimation noise behaves like prediction noise.

Fortunately, however, risk functions tend to be well-behaved when it comes to noise in a way that we can exploit to eliminate risk inflation. If the risk function is analytic in a neighborhood around $\widetilde{Y}$, then 
\begin{align}
    \E \bracket[\Big]{\phi \paren[\Big]{Y, \widetilde{Y} + \frac{SZ}{\sqrt{m}}} \,\Big|\, Y, \widetilde{Y}}
    = \E \bracket[\big]{\phi \paren{Y, \widetilde{Y}} \,\big|\, Y, \widetilde{Y}} +
    \frac{1}{2m} \E \bracket[\big]{S^2 \phi'' \paren{Y, \widetilde{Y}} \,\big|\, Y,\widetilde{Y}} + o\paren[\Big]{\frac{1}{m}},
\end{align}
where $o(t)/t \to 0$ as $t \to 0$, here at a rate that may depend on $Y$ and $\widetilde{Y}$. Provided these quantities are well behaved over $Y$ and $\widetilde{Y}$, we can take the expectation over these variables as well, obtaining the large-sample limit plug-in risk estimate as a function of $m$:
\begin{align}
    R(m) \defeq \E \bracket[\Big]{\phi \paren[\Big]{Y, \widetilde{Y} + \frac{SZ}{\sqrt{m}}}}
    = \underbrace{\E \bracket[\big]{\phi \paren{Y, \widetilde{Y}}}}_{\rdefeq R_0}
    {}+ \frac{1}{m} \underbrace{\E \bracket[\Big]{\frac{S^2}{2} \phi'' \paren{Y, \widetilde{Y}}}}_{\rdefeq R_0'} {}+ o\paren[\Big]{\frac{1}{m}}.
\end{align}
Note that $R_0$ is the true quantity to be estimated, and that for sufficiently large $m$ we have the approximate linear relation $R(m) \approx R_0 + R_0' / m$. Therefore, we can estimate $R_0$ by evaluating our plug-in estimate $R(m)$ at several choices of $m$ and extracting the intercept term using linear regression. To avoid the need to compute additional Jacobian--vector products, we propose to take subsamples of size $m' \leq m$ of the individual diagonal estimates $(\widetilde{\mJ} \vw_k) \odot \vw_k$ which have already been computed, such that obtaining the averaged BKS estimate for the subsample has negligible cost. Repeating this subsampling procedure for several sufficiently large values of $m'$ (e.g., $m' \in \set{m/2, \ldots, m}$), we obtain our complete procedure, which we call RandALO, described in \Cref{alg:alo-randomized} and demonstrated in \Cref{fig:lasso-bks-convergence}.

\begin{figure}[t]
    \centering
    \includegraphics[width=\textwidth]{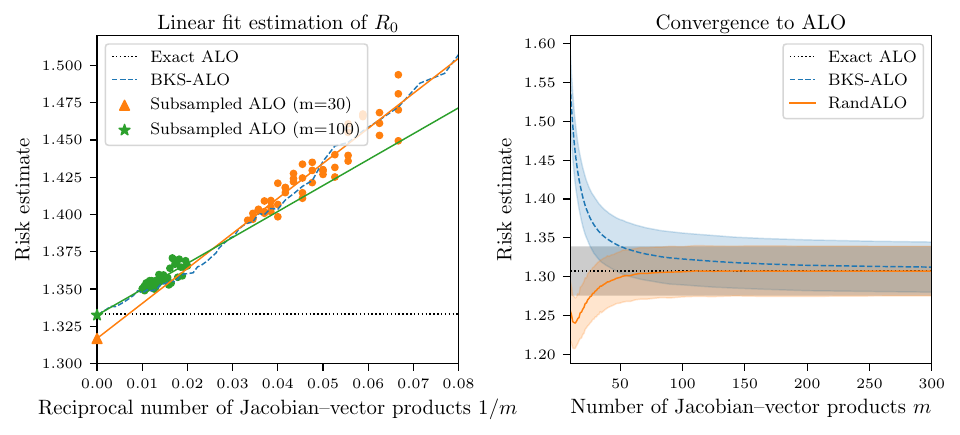}
    \caption{\textbf{Left:} Although the plug-in estimation of ALO using BKS diagonal estimation (blue, dashed) is significantly biased, by evaluating the plug-in BKS risk estimate with subsampled Jacobian--vector products (dots), we can obtain high quality debiased estimates (triangle, star) of ALO using a linear regression (solid lines).
    \textbf{Right:} Our complete procedure RandALO (orange, solid) converges very quickly in $m$ to the limiting ALO risk estimate, which provides an accurate estimate of test error (black, dotted). It converges significantly more quickly than the na\"ive plug-in BKS estimate (blue, dashed). Lines and shaded areas denote mean and standard deviation over 100 random trials of a lasso problem with $n = p = 5000$ described in \Cref{sec:num:lasso}}.
    \label{fig:lasso-bks-convergence}
\end{figure}

\begin{algorithm}[H]
\caption{Perform randomized ALO risk estimation}
\label{alg:alo-randomized}
\begin{algorithmic}[1] 
\Procedure{RandALO}{$\vy \in \reals^n, \widehat{\vy} \in \reals^n$, $\mJ = \frac{\partial \widehat{\vy}}{\partial \vy} \in \reals^{n \times n}$, $\ell \colon \reals^2 \to \reals$, $\phi \colon \reals^2 \to \reals$, $m \in \nats$}
    \State $\tilde{J}_{ij} \gets - J_{ij} \cdot \frac{\ell''(y_j, \hat{y}_j)}{\frac{\partial \ell'(y_j, \hat{y}_j)}{\partial y_j}}$.
    \Comment{Normalize Jacobian.}
    \State Sample $\mW \in \reals^{n \times m}$ as $W_{ij} \overset{\iid}\sim \mathrm{Rademacher}$.
    \State $\mD \gets (\widetilde{\mJ} \mW) \odot \mW \in \reals^{n \times m}$. 
    \Comment{Randomized diagonal estimation.}
    \State $\mu_i \gets \frac{1}{m} \sum_{j=1}^m D_{ij}$.
    \State $\sigma_i^2 \gets \frac{1}{m - 1} \sum_{j=1}^m (D_{ij} - \mu_i)^2$. 
    \Comment{Compute statistics using all samples.}
    
    \For{$m' \in \set{\frac{m}{2}, \ldots, m}$}
    \Comment{Iterate over subsets of different sizes.}
    \State Sample random subset $\setM \subseteq [m]$ of size $m'$.
    \State $\hat{\mu}_i \gets \frac{1}{m'} \sum_{j \in \setM} D_{ij}$.
    \Comment{Compute subset statistics.}
    \State $d_i \gets \textsc{TruncatedNormalMean}(\hat{\mu}_i, \frac{\sigma_i}{\sqrt{m'}}, 0, 1)$
    \Comment{Correct noise outside of $[0, 1]$ bounds.}
    \State $\tilde{y}_i \gets \hat{y}_i + \frac{\ell'(y_i, \hat{y}_i)}{\ell''(y_i, \hat{y}_i)} \frac{d_i}{1 - d_i}$.
    \Comment{Estimate LOO prediction.}
    \State $\hat{R}(m') \gets \frac{1}{n} \sum_{i=1}^n \phi(y_i, \tilde{y}_i)$.
    \Comment{Compute risk estimate for subset.}
    \EndFor

    \State $\hat{R}_0, \hat{R}_0' \gets$ Regress $\hat{R}(m') \sim R_0 + \frac{R_0'}{m'}$.
    \Comment{Extrapolate limiting risk.}
    \State \textbf{return} $\hat{R}_0$.
\EndProcedure
\end{algorithmic}
\end{algorithm}

\section{Computing Jacobian--vector products}\label{sec:jvp}

We have thus far considered the computation of Jacobian--vector products as a black box operation available to \Cref{alg:alo-randomized}. In some cases this may be preferred, such as when the Jacobian can be computed via automatic differentiation, as recently exploited in a similar setting for computing Stein's unbiased risk estimate \citep{nobel2023tractable}.
However, for linear models as we focus on in this work, it is important that the Jacobian--vector products have similar computational complexity to the efficient iterative algorithms that practitioners generally use to obtain their predictive models.

We can determine the Jacobian $\partial \widehat{\vy}/\partial \vy$ using implicit differentiation of \cref{eq:erm} to obtain the expression in \cref{eq:alo} when the regularizer is twice differentiable, but in the general case of non-differentiable regularizers, subtleties arise that need to be handled carefully.
We consider a simple yet powerful case, where $r$ can be written as a sum of functions  with at most one point of non-differentiablity precomposed with affine transformations.
Computing Jacobian--vector products then reduces to solving an equality-constrained quadratic program, as stated in the following theorem.

\begin{theorem}
\label{thm:derivative-as-qp}
Let $\widehat{\vy} = \mX \widehat{\vbeta}$ for $\widehat{\vbeta}$ solving \cref{eq:erm}, let $\ell$ be twice-differentiable and strictly convex in its second argument, and 
let the regularizer have the form $r(\vbeta) = \sum_{k=0}^K r_k(\mA_k \vbeta + \vc_k)$ for some $\mA_k \in \reals^{q_k \times p}$, $\vc_k \in \reals^{q_k}$, and convex $r_k: \reals^{q_k} \to \reals$ such that
\begin{itemize}
\item $r_0$ is twice-differentiable everywhere,
\item for all $k \in \{1, 2, \ldots, K\}$, $r_k$ is twice-differentiable everywhere except $\vzero$, 
and
\item $\setS = \{0\} \cup \{k\in \{1, 2, \ldots, K\}: \mA_k\widehat\vbeta + \vc_k \neq \vzero\}$ is constant in a neighborhood of $\vy$.
\end{itemize}
Then for any vector $\vz \in \reals^n$, if there is a unique solution $\vv^* \in \reals^p$ to the quadratic program 
\begin{equation}
\begin{array}{ll}
\mathrm{minimize}_\vv & \frac{1}{2} \vv^\top \left(\mX^\top \mH_\ell \mX + \sum_{k\in\setS} \mA_k^\top \nabla^2 r_k(\mA_k \widehat{\vbeta} + \vc_k) \mA_k \right) \vv - \vv^\top \mX^\top \mH_\ell \vz \\
\mathrm{subject~to} & \mA_k \vv = \vzero, \qquad k \notin \setS,
\end{array}
\label{eq:jvp-qp}
\end{equation}
then $\widetilde{\mJ} \vz = \mX \vv^*$.
\begin{proof}
First we show that the uniqueness of $\vv^\star$ implies that a particular matrix is invertible.
Note that the optimality conditions of \cref{eq:jvp-qp} are equivalent to 
\begin{equation}\label{eq:kkt-jvp-qp}
\begin{array}{l}
\left(\mX^\top \mH_\ell\mX+ \sum_{k\in\setS} \mA_k^\top \nabla^2 r_k(\mA_k \widehat{\vbeta} + \vc_k) \mA_k \right) \vv + \mN^\top \vnu =  \mX^\top \mH_\ell \vz \\
\mN \vv = \vzero,
\end{array}
\end{equation}
where $\mN$ is a full-rank matrix whose null space equals the intersection of the null spaces of $\mA_k$ for $k \notin\setS$.
For convenience, let $\mP = \mX^\top \mH_\ell\mX+ \sum_{k\in\setS} \mA_k^\top \nabla^2 r_k(\mA_k \widehat{\vbeta} + \vc_k) \mA_k$.
Rewritten more compactly, \cref{eq:kkt-jvp-qp} is the system of linear equations
\begin{equation}\label{eq:linear-system-kkt-qp}
\begin{bmatrix}
\mP & \mN^\top \\
\mN & \vzero
\end{bmatrix}
\begin{bmatrix} \vv \\ \vnu \end{bmatrix}
= \begin{bmatrix}
\mX^\top \mH_\ell\vz \\ \vzero
\end{bmatrix}.
\end{equation}

Seeking a contradiction, assume that there exists $\vu \neq \vzero$ such that $\mN \vu = \vzero$  and $\vu^\top \mP \vu = \vzero$.
Then $\vv^\star + \vu$ solves \cref{eq:linear-system-kkt-qp}.
This would imply that $\vv^\star$ is not the unique solution to \cref{eq:jvp-qp}, contradicting our assumption that $\vv^*$ is unique.
Therefore the matrix on the LHS is invertible~\cite[\S10.1.1]{bv2004convex}.


We now turn to our main focus: applying the Implicit Function Theorem.
By rewriting \cref{eq:erm} with new constraints $\vd_k = \mA_k \vb + \vc_k$ for $k \notin \setS$, we obtain
\begin{equation}\label{eq:erm-r-as-sum}
\widehat\vbeta = \argmin_{\vb: k \notin \setS: \mA_k \vb + \vc_k = \vd_k} \sum_{i=1}^n \ell(y_i, \vx_i^\top \vb) + 
\sum_{k\in \setS} r_k(\mA_k \vb + \vc_k)+
\sum_{k\notin\setS} r_k(\vd_k).
\end{equation}
For convenience, let $L(\vy, \mX\vb) = \sum_{i=1}^n \ell(y_i, \vx_i^\top \vb)$ and $\nabla L(\vy, \mX\vb)$ be the gradient of $L$ with respect to its second argument.
In the neighborhood of $\vy$ for which $\setS$ is constant, this mapping from $\vy$ to $\widehat\vbeta$ is equal to the mapping from $\vy$ to $\widetilde\vbeta$ given by 
\begin{equation}\label{eq:erm-r-as-sum-over-S}
\widetilde\vbeta = \argmin_{\vb: (k \notin \setS: \mA_k \vb + \vc_k = \vzero)} \sum_{i=1}^n \ell(y_i, \vx_i^\top \vb) + 
\sum_{k\in \setS} r_k(\mA_k \vb + \vc_k)+
 \sum_{k\notin\setS} r_k(\vzero).
\end{equation}
The constraints $k\notin\setS: \mA_k\vb + \vc_k = \vzero$ form an affine set parallel to the intersection of the null spaces of $\mA_k$ for all $k \notin \setS$.
Accordingly, we can write the optimality conditions of \cref{eq:erm-r-as-sum-over-S} as
\begin{equation}\label{eq:kkt}
\begin{array}{ll}
    \mX^\top \nabla L(\vy, \mX\vb)
    + \sum_{k\in \setS} \mA_k^\top \nabla r_k(\mA_k\vb + \vc_k)
    + \mN^\top \vlambda = \vzero, \\
     \mN \vb - \mN\widehat\vbeta = \vzero,
\end{array}
\end{equation} 
where $\mN$ is as defined above.
In order to apply the Implicit Function Theorem \cite[Theorem 9.28]{rudin76} to \cref{eq:kkt},
we note that there exists an open set $\setE$ around $\widehat\vbeta$, such that for all $\vb \in \setE$, $\mA_k\vb + \vc_k \neq \vzero$ for all $k \in \setS$.
Therefore, the left-hand side is continuously differentiable on $\setE$.
Further, the Jacobian of the left-hand sides of \cref{eq:kkt} with respect to $\vb, \vlambda$ is given by
\begin{equation}\label{eq:jacobian-for-ift}
\mM = \begin{bmatrix}
\mP  & \mN^\top \\
\mN & \vzero \\
\end{bmatrix}.
\end{equation}
This matrix is the same as the left-hand side of \cref{eq:linear-system-kkt-qp}, which we showed to be invertible above. 

Accordingly, the Implicit Function Theorem guarantees that $\widetilde\vbeta$ and $\vlambda$ are differentiable functions of $\vy$ whose derivatives are given by 
\begin{equation}\label{eq:ift}
\begin{bmatrix}
\frac{\partial \widetilde\vbeta}{\partial \vy} \\
\frac{\partial \vlambda}{\partial \vy} \\
\end{bmatrix}
=
-\mM^{-1} \begin{bmatrix}
\mX^\top \frac{\partial^2 L}{\partial \mX\vb \partial \vy}\\
\vzero \\
\end{bmatrix}.
\end{equation}
Since $\widetilde\vbeta = \widehat\vbeta$ on the neighborhood where $\setS$ is constant, we can conclude that
$  \frac{\partial \widetilde\vbeta}{\partial \vy} = \frac{\partial \widehat\vbeta}{\partial \vy}$.
Next, recall that $\widetilde\mJ\vz = -\mJ\left(\frac{\partial^2 L}{\partial\mX\vb \partial\vy}\right)^{-1} \mH_\ell\vz = \mX\frac{\partial\widehat\vbeta}{\partial \vy} \Big(-\left(\frac{\partial^2 L}{\partial\mX\vb \partial\vy}\right)^{-1} \Big) \mH_\ell\vz  = \mX \vv^\star$.
Basic algebra applied to \cref{eq:ift} shows that 
\begin{equation}\label{eq:ift-rearranged}
\mM\begin{bmatrix}
\vv^\star \\
\vnu^\star \\
\end{bmatrix}
=
\begin{bmatrix}
\mX^\top \mH_\ell \vz\\
\vzero \\
\end{bmatrix}
\end{equation}
which is simply \cref{eq:linear-system-kkt-qp}.
\end{proof}
\end{theorem}

The conditions of this theorem are somewhat technical, but they allow us to generalize to many common regularizers of interest.
The essence of the result is that when a solution is found at a non-differentiable point of a regularizer, the Jacobian has no component in the corresponding direction.
We now give several examples of how to apply this theorem for a few popular regularizers.

\begin{corollary} 
\label{cor:lasso-elastic-net}
For the elastic net penalty $r(\vbeta) = \frac{\lambda}{2} \norm[2]{\vbeta}^2 + \theta \norm[1]{\vbeta}$, the Jacobian has the form
\begin{align}
    \widetilde{\mJ} = \mX_\setS \inv{\mX_\setS^\transp \mH_\ell \mX_\setS + \lambda \mI} \mX_\setS^\transp \mH_\ell,
    \label{eq:en-penality-jacobian}
\end{align}
where $\mX_\setS$ selects only the columns of $\mX$ from the set $\setS = \set{j \colon \hat{\beta}_j \neq 0}$, when $\setS$ is locally constant and either $\lambda >0$ or $\mX_\setS^\top \mH_\ell\mX_\setS$ is invertible.
\begin{proof}
Apply \cref{thm:derivative-as-qp} with $r_0(\vbeta) = \frac{\lambda}{2} \norm[2]{\vbeta}^2$ and $r_k(\vbeta) = \theta |\ve_k^\top \vbeta|$ for $k \in [p]$, where $\ve_k \in \reals^p$ are the standard basis vectors, to obtain that $\widetilde\mJ \vz = \mX \vv^\star$ such that $\vv^\star$ satisfies
\begin{equation}
\begin{array}{ll}\label{eq:qp-for-en-jvp}
\text{minimize} & \frac{1}{2} \vv^\top \left(\mX^\top \mH_\ell\mX + \lambda \mI\right)\vv - \vv^\top \mX^\top \mH_\ell \vz \\
\text{subject to} & \ve_k^\top \vv = 0 \; : \; k \notin\setS.
\end{array}
\end{equation}
The constraints ensures that the columns of $\mX$ associated with $k \notin\setS$ are multiplied by $0$, allowing the $\vv_\setS^\star$ to be expressed as the minimization of the following quadratic: $\frac{1}{2} \vv_\setS^\top\left(\mX_\setS^\top \mH_\ell \mX_\setS + \lambda \mI\right)\vv_\setS - \vv_\setS^\top \mX_\setS^\top \mH_\ell \vz$.
This can be solved analytically to show that $\widetilde\mJ\vz = \mX_\setS \inv{\mX_\setS^\transp \mH_\ell \mX_\setS + \lambda \mI} \mX_\setS^\transp \mH_\ell \vz $.
Because either $\mX_\setS^\transp \mH_\ell \mX_\setS$ is invertible or $\lambda > 0$, such that $\mX_\setS^\transp \mH_\ell \mX_\setS + \lambda \mI$ is invertible, $\vv^*$ is unique for any $\vz$ as required by the theorem.
By considering $\vz = \ve_i$ for $i \in [n]$, we obtain \cref{eq:en-penality-jacobian}. 
%
%
\end{proof}
\end{corollary}
The above corollary of course recovers the standard ridge regression ($\theta = 0$) and lasso ($\lambda = 0$) penalty Jacobians as special cases and matches the extension of ALO to the elastic net by \citet{auddy2024approximate}. Other separable penalties such as the $\ell_p^p$ norm admit a similar form. However, we are not restricted to separable penalties. For example, we can also pre-transform $\vbeta$ before applying a separable penalty.

\begin{corollary}
\label{cor:lin-lasso}
For the linearly transformed $\ell_1$ penalty $r(\vbeta) = \lambda \norm[1]{\mA^\transp \vbeta}$, the Jacobian--vector product $\widetilde{\mJ} \vz = \mX\vv^*$ where $\vv^*$ is the unique optimal solution to 
\begin{equation}
\label{eq:lin-lasso-qp}
\begin{array}{ll}
\textrm{minimize}_\vv & \frac{1}{2} \vv^\top \mX^\top \mH_\ell \mX  \vv - \vv^\top \mX^\top \mH_\ell \vz \\
\textrm{subject to} & \mA_\setSbar^\transp \vv = \vzero,
\end{array}
\end{equation}
where $\mA_\setSbar$ selects only the columns $\va_j$ of $\mA$ from the set $\setSbar = \set{j \colon  \va_j^\transp \widehat{\vbeta} = 0}$ when it is locally constant.
\begin{proof}
Apply \cref{thm:derivative-as-qp} with $r_0(\vbeta) = 0$ and $r_k(\vbeta) = \lambda |\va_k^\transp \vbeta|$ for $k \in [p]$.
\end{proof}
\end{corollary}

The above penalty is commonly used in the compressed sensing literature, where $\mA$ transforms $\vbeta$ into a frame in which it should be sparse. Another non-separable example is the group lasso.

\begin{corollary}
\label{cor:group-lasso}
For the group lasso $r(\vbeta) = \sum_{k=1}^K \lambda \norm[2]{\mPi_k \vbeta}$ with disjoint idempotent subspace projection operators $\mPi_k \in \reals^{p \times p}$ such that $\mPi_k \mPi_{k'} = \vzero$ and $\sum_{k=1}^K \mPi_k = \mI$, the Jacobian has the form
\begin{align}
    \widetilde{\mJ} = \mX \mPi_\setS \Bigg(\mPi_\setS \mX^\transp \mH_\ell \mX \mPi_\setS + \sum_{k \in \setS} \frac{\lambda}{\norm[2]{\mPi_k \vbetahat}} \Bigg( \mPi_k - \frac{\mPi_k \vbetahat \vbetahat^\transp \mPi_k}{\norm[2]{\mPi_k \vbetahat}^2} \Bigg)  \Bigg)^\dagger \mPi_\setS \mX^\transp \mH_\ell,
\end{align}
where $\mPi_\setS = \sum_{k \in \setS} \mPi_k$ for $\setS = \set{k \colon \norm[2]{\mPi_k \vbetahat} \neq \vzero}$ when $\setS$ is locally constant. 
\begin{proof}
The Hessian of a single $\ell_2$ norm component $r_k(\vbeta) = \lambda \norm[2]{\mPi_k \vbeta}$ is given by
\begin{align}
    \nabla^2 r_k(\vbeta) = \frac{\lambda}{\norm[2]{\mPi_k \vbetahat}} \Bigg( \mPi_k - \frac{\mPi_k \vbetahat \vbetahat^\transp \mPi_k}{\norm[2]{\mPi_k \vbetahat}^2} \Bigg).
\end{align}
We can then apply \cref{thm:derivative-as-qp}. The linear constraint becomes $\mPi_k \vv = \vzero$ for all $k$ such that $\mPi_k \vbetahat = \vzero$, which is equivalent to restricting the system to the complementary subspace defined by $\mPi_\setS$.
\end{proof}
\end{corollary}
The standard group lasso where feature are partitioned into groups follows under the above corollary when $\mPi_k$ are diagonal matrices with $1$'s indicating feature membership in each group and $0$'s elsewhere. We give the non-overlapping group lasso here as it has a simple closed-form Jacobian, but overlapping groups can be easily accommodated in the quadratic program formulation.

Since \cref{thm:derivative-as-qp} describes the Jacobian as a quadratic program in $p$-dimensional space, it requires a known feature representation of the data. However, we can still consider linear models in unknown feature spaces for the ridge regularizer, enabling us to apply RandALO to kernel methods.

\begin{corollary}
\label{cor:kernel}
For the ridge penalty $r(\vbeta) = \frac{\lambda}{2} \norm[2]{\vbeta}^2$, the Jacobian admits a formulation in terms of the kernel matrix $\mK = \mX \mX^\transp$. If $\mK$ and $\mH_\ell$ are invertible, then
\begin{align}
    \widetilde{\mJ} = \mK \inv{\mK + \lambda \mH_\ell^{-1}}.
\end{align}
\begin{proof}
Without loss of generality, since $\mK$ is invertible, assume that $p = n$. Then starting from the ridge penalty solution in \Cref{cor:lasso-elastic-net} with $\theta = 0$ and $\setS = [n]$, we can introduce $\mX$ inside and outside the inverse to obtain
\begin{align}
    \widetilde{\mJ} = \mX \mX^\transp \inv{\mX \mX^\transp \mH_\ell \mX \mX^\transp + \lambda \mX \mX^\transp} \mX \mX^\transp \mH_\ell.
\end{align}
Bringing the $\mX \mX^\transp \mH_\ell$ on the right inside the inverse on the left as $\mH_\ell^{-1} \mK^{-1}$ cancels terms to produce the stated expression. 
\end{proof}
\end{corollary}


\subsection{Alternative approaches}

There are other methods to evaluate the Jacobian--vector products.
By reducing a convex data fitting problem to the conic form of the optimization problem, the Jacobian--vector products can be evaluated as a system of linear equations  \citep{diffcp2019} or via perturbations to the input solver \citep{paulus2024lpgd}.
Other work enables the minimization of non-convex least squares problems \citep{pineda2022theseus}.
We leave extending our efficient method to particular cases of interest such as non-differentiable losses (even jointly with non-differentiable regularizers) to future work.

\section{Numerical Experiments}
\label{sec:num}

We now demonstrate the effectiveness of RandALO on a variety of problems, 
including with different losses, regularizers, and models on the data. In particular, we show how RandALO is able to outperform $K$-fold CV (typically, $5$-fold CV as a single point of comparison) in terms of both accuracy of the risk estimate and computational cost.
Unless otherwise specified, we implemented these experiments as follows.

\paragraph{RandALO implementation.}
We wrote an open-source Python implementation of RandALO based on PyTorch \citep{paszke2019pytorch} available on PyPI as \texttt{randalo} and at \url{https://github.com/cvxgrp/randalo}, capable of accepting arbitrary black-box implementations of Jacobian--vector products.

\paragraph{Jacobian--vector product implementation.}
We implement our Jacobian-vector products using the torch\_linops library (available at \url{https://github.com/cvxgrp/torch_linops}) and SciPy \citep{2020SciPy-NMeth} to solve the quadratic program in \cref{eq:jvp-qp}.
For problems with dense data and $n < 5000$, we compute the solution directly via the LDL factorization, for sparse data we apply the conjugate gradient (CG) method, and for large dense data we apply MINRES.

\paragraph{Machine learning implementation.}
We used standard and well-optimized methods for fitting common models provided by scikit-learn \citep{scikit-learn} for lasso, logistic regression, and cross-validation.
For first-difference $\ell_1$-regularized regression we implemented the solution using the Clarabel solver \citep{goulart2024clarabel} through CVXPY~\citep{diamond2016cvxpy,agrawal2018rewriting}.
For kernel logistic regression, we implement a Newton's method solver using CG.

\paragraph{Hyperparameter selection.} 
We consider high-dimensional problems similar to those for which ALO is known to provide consistent risk estimation \citep{xu2021consistent}. For boxplots, we select hyperparameters roughly of the same order as the optimal parameter given the noise level such that the risks at those hyperparameters are of interest.

\paragraph{Risk metrics.} We consider the regression risk metric of squared error $\phi(y, z) = (y - z)^2$ as well as the classification risk metric of misclassification error $\phi(y, z) = \ind\set{yz < 0}$. For each risk estimate $\hat{R}$, we report the relative error $(\hat{R} - R(\widehat{\vbeta}))/R(\widehat{\vbeta})$ from the conditional risk $R(\widehat{\vbeta}) = \expect{\phi(y, \vx^\transp \widehat{\vbeta}) \,|\, \widehat{\vbeta}}$, with boxplots depicting 100 random trials.

\paragraph{Relative time.} All relative times are reported with respect to the time to fit the model $\widehat{\vbeta}$ according to \cref{eq:erm} on the full training data. For CV, the time reported includes only the fitting of the models $\widehat{\vbeta}_{-\setP}$ in \cref{eq:cv}, and computing $\widehat{\vbeta}$ after model selection would be an additional computational cost. For ALO methods, we include the original fitting of the model $\widehat{\vbeta}$ and add the time required to run \Cref{alg:alo-randomized} (omitting the inflation debiasing step for BKS-ALO).

\paragraph{Compute environment.} 
We compute on Stanford's Sherlock cluster, drawing new random data and BKS vectors each trial. We run each trial on a single core with 16GB of memory, reporting wall-clock times for each computation done. 

\subsection{Lasso: Efficiency and problem scaling}
\label{sec:num:lasso}

\begin{figure}[t]
    \centering
    \includegraphics[width=\textwidth]{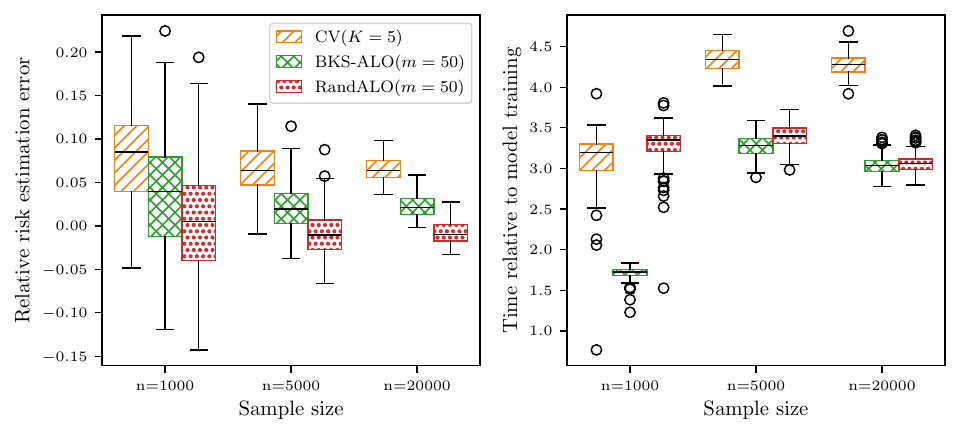}
    \caption{\textbf{Left:} For a lasso problem in proportionally high dimensions $p = n$, CV suffers from bias that does not vanish with $n$ even as risk concentrates. Meanwhile, even BKS-ALO with its biased risk estimate at only $m = 50$ Jacobian--vector products is more accurate than CV at lower computational cost (\textbf{right}). Going further, RandALO removes the bias for the same choice of $m$ with a computational overhead that vanishes as $n$ increases. }
    \label{fig:lasso-scaling-normal}
\end{figure}

In this experiment, we generate $\vx_i \sim \normal(\vzero, \mI_p)$ and $y_i = \vx_i^\transp \vbeta^* + \epsilon_i$ for $\vbeta^*$ having $s=p/10$ non-zero elements drawn as i.i.d. $\normal(0, 1/s)$ and $\epsilon_i \sim \normal(0, 1)$. 
We fit a lasso model with $\ell(y, z) = \tfrac{1}{2} (y - z)^2$ and $r(\vbeta) = \lambda \norm[1]{\vbeta}$.
We fix $n = p$ and set $\lambda = \sqrt{n}$, which balances the loss and regularizer to be of the same order. The conditional squared error risk is given by
\begin{align}
    R(\widehat{\vbeta}) = \expect{\phi(y, \vx^\transp \widehat{\vbeta}) \,|\, \widehat{\vbeta}}= \norm[2]{\widehat{\vbeta} - \vbeta^*}^2 + 1.
    \label{eq:cond-risk-squared-error}
\end{align}

In \Cref{fig:cv-tradeoff} we consider this problem for $n = p = 5000$ using the direct dense solver for Jacobian--vector products. Even with the highly optimized coordinate descent solver for lasso, the estimation error--computation trade-off curve for CV is entirely dominated by both BKS-ALO and RandALO. For low numbers of Jacobian--vector products $m$, RandALO provides roughly an order of magnitude improvement in risk estimation and a small additional cost. For larger $m$, RandALO does converge more quickly to the limiting risk estimate, but both BKS-ALO and RandALO are within the standard deviation of the conditional risk from the marginal risk, which is what CV methods provide an estimate of, and so either method could be equally trusted.
We consider the same problem setup in  \Cref{fig:lasso-bks-convergence} where we demonstrate the debiasing improvement of RandALO over BKS-ALO.

In \Cref{fig:lasso-scaling-normal}, we consider increasing values of $n$. We use MINRES to compute the Jacobian--vector products for $n > 1000$ since it scales better to large $n$.
We see that the inconsistency of CV remains the same with increasing problem dimension without vanishing, while with only $m=50$ Jacobian--vector products, we have virtually eliminated all bias in RandALO at only a fraction of the time of $5$-fold CV.
The overhead due to applying the debiasing procedure of RandALO is substantial for $n = 1000$, but it becomes negligible as $n$ increases.



\subsection{Sparse first-difference regression}

In this experiment, we demonstrate the dramatic improvement of randomized ALO over CV for non-standard problems. We now generate data in the same manner as in \Cref{sec:num:lasso} except we let $\vbeta^*$ be the cumulative sum of $\vb^* \in \reals^p$ having $s = p / 10$ non-zero elements with indices selected at random without replacement and values drawn i.i.d. $\normal(0, 2 / sp)$, such that $\vbeta^*$ is piecewise constant $\expect{\norm[2]{\vbeta^*}^2} = 1$,
and we let $\epsilon_i \sim \normal(0, 0.01)$. 
Accordingly, we use the first-differences regularizer $r(\vbeta) = \lambda \|\mD \vbeta\|_1$ where 
\[
\mD = 
\begin{bmatrix}
-1 & 1 & 0 & \cdots & 0 \\
0 & -1 & 1 & \cdots & 0 \\
\vdots & \vdots & \vdots & \ddots & \vdots \\
0 & 0 & \cdots & -1 & 1 \\
\end{bmatrix}.
\]
This regularizes $\widehat\vbeta$ towards being piecewise constant.
We fix $n = p$ and $\lambda = p$.
Unlike the lasso, this regularizer is used much more rarely and lacks the same highly specialized solving methods.
This results in significantly slower solve times and, as shown in \Cref{fig:first-diff-scaling-normal}, being able to avoid the expense of having to fit multiple models, instead solving the quadratic program in \cref{eq:lin-lasso-qp} for Jacobian--vector products, yields much improved runtime using RandALO at no cost to accuracy.

\begin{figure}[t]
    \centering
    \includegraphics[width=\textwidth]{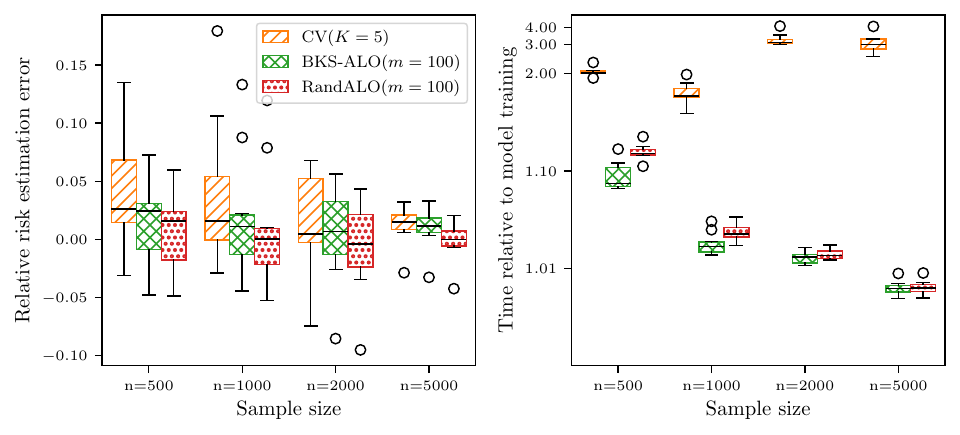}
    \caption{
    On the more computationally involved problem of sparse first-differences, randomized ALO provides similar statistical improvements over $5$-fold CV as in the lasso problem but dramatically improves computationally. The Jacobian--vector products for ALO require only the solving of the quadratic program in \cref{eq:lin-lasso-qp}, while CV must repeatedly solve a much more difficult convex optimization problem.
    Relative time on the $y$-axis is plotted on a $\log(y - 1)$ scale to emphasize the minimal additional cost of RandALO after the model training. For this experiment, we report box plots for only 10 trials, since fitting models at the largest scales takes a few hours per trial for $5$-fold CV.}
    \label{fig:first-diff-scaling-normal}
\end{figure}

\subsection{Logistic ridge regression}
\label{sec:logistic}

In this example, we demonstrate that randomized ALO works outside of squared error regression. We consider a binary classification problem using the logistic loss function $\ell(y, z) = \log(1 + e^{-yz})$ used for binary classification regularized with the ridge penalty $r(\vbeta) = \tfrac{\lambda}{2}\norm[2]{\vbeta}^2$.
We again consider $\vx_i \sim \normal(\vzero, \mI_p)$, while for labels we let $y_i = 1$ with probability $\sigma(\rho \vx_i^\transp \vbeta*)$ for $\sigma(u) = 1 / (1 + e^{-u})$, $\rho = 5$, and $\vbeta^*$ having $s = p / 4$ nonzero elements drawn as i.i.d.\ $\normal(0, 1/s)$, and let $y_i = -1$ otherwise.
We let $n = 10000$ and $p = 4000$ and choose $\lambda = n$.
We evaluate misclassification error with risk function $\phi(y, z) = \ind\set{yz < 0}$, which has the conditional risk
\begin{align}
    R(\widehat{\vbeta}) = \expect{\phi(y, \vx^\transp \widehat{\vbeta}) \,|\, \widehat{\vbeta}}
    = \expect{\sigma(-\mathrm{sgn}(\widehat{Z}) \rho Z)}
    \quad \text{for} \quad
    \begin{bmatrix}
        Z \\
        \widehat{Z}
    \end{bmatrix}
    \sim \normal \Big( \vzero,
    \begin{bmatrix}
        \vbeta^{*\transp} \\
        \widehat{\vbeta}^\transp
    \end{bmatrix}
    \begin{bmatrix}
        \vbeta^* & 
        \widehat{\vbeta}
    \end{bmatrix}
    \Bigg),
\end{align}
which we can compute via numerical integration with the 2-dimensional Gaussian density. We emphasize that this risk is distinct from the logistic loss used to fit the model.
For this problem, we have slower solvers available than for lasso, making the overhead for randomized ALO using the direct solver for Jacobian--vector products small compared to $5$-fold CV, as shown in \Cref{fig:comp-all} (left).

\begin{figure}[t]
    \centering
    \includegraphics[width=\textwidth]{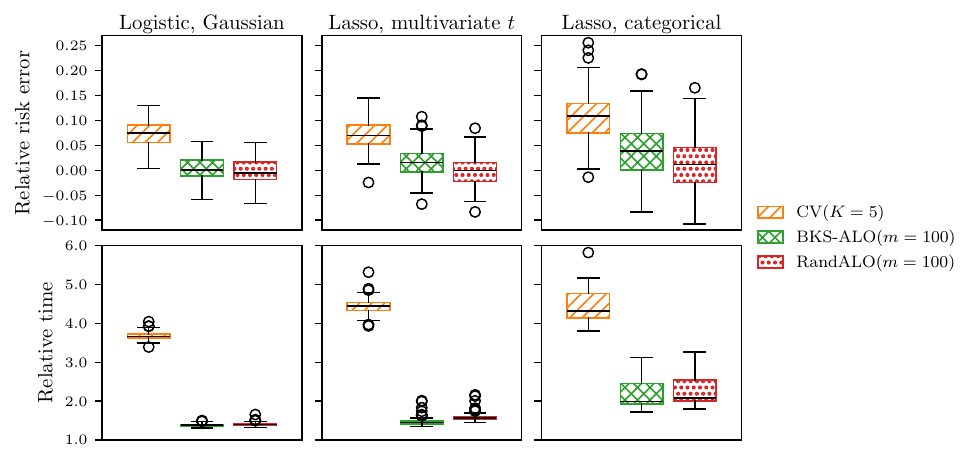}
    \caption{RandALO provides consistent risk estimation and outperforms CV on a variety of problems and data types beyond least squares regression and Gaussian data. Here we show logistic regression with ridge penalty on Gaussian data (\textbf{left}) as well as lasso on multivariate $t$ elliptical data (\textbf{middle}) and categorical data (\textbf{right}). 
    }
    \label{fig:comp-all}
\end{figure}

\subsection{Multivariate \texorpdfstring{$t$}{t} data}

In this example we consider our first departure from known guarantees for the consistency of ALO. We consider the same setting as \Cref{sec:num:lasso} except we let the data be drawn from a scaled multivariate $t$-distribution. That is, $\vx_i \sim \sqrt{t_i} \vz_i$ where $\vz_i \sim \normal(\vzero, \mI_p)$ and $(\nu - 2) / t_i \sim \chi_\nu^2$ where $\chi_\nu^2$ denotes the chi-squared distribution with $\nu$ degrees of freedom. We let $\nu = 5$. Since $\expect{\vx \vx^\transp} = \mI_p$, we have the same expression for conditional risk from \cref{eq:cond-risk-squared-error} as in the Gaussian case. 

Notably, because of the instance-wise scalars $t_i$, the diagonal Jacobian elements $\tilde{J}_{ii}$ do not all concentrate to the same value, meaning that GCV, which uses $\tfrac{1}{n} \tr[\widetilde{\mJ}]$ where ALO uses $\tilde{J}_{ii}$ in \cref{eq:alo}, cannot be consistent in general. However, as shown in \Cref{fig:comp-all} (middle), ALO provides consistent risk estimation and can be implemented efficiently using our randomized method.


\subsection{Categorical data}
In our next example, we go even further from standard random matrix assumptions. We sample $\vx_i$ by drawing $i_1, \ldots, i_d$ independently and uniformly from $[k]$ and concatenating standard basis vectors to form $\vx_i = \sqrt{k}[\ve_{i_1}^\transp \ldots \ve_{i_d}^\transp]^\transp \in \reals^p$ where $p = dk$. We then generate $y_i = \vx_i^\transp \vbeta^* + \epsilon_i$ for $\vbeta^*$ having $s = p/10$ non-zero elements drawn as i.i.d.\ $\normal(0, 1/s)$ and $\epsilon_i \sim \normal(0, 1/2)$. We choose $n = d = 2000$ and $k = 10$ and generate $\mX$ as a sparse data matrix. We apply lasso with $\lambda = \sqrt{d}$. For this problem, we have the covariance structure
\begin{align}
    \expect{\vx \vx^\transp} = \begin{bmatrix}
        \mI_k & \tfrac{1}{k}\vone_k & \ldots & \tfrac{1}{k}\vone_k \\
        \tfrac{1}{k}\vone_k & \mI_k & \ldots & \tfrac{1}{k}\vone_k \\
        \vdots & \vdots & \ddots & \vdots \\
        \tfrac{1}{k}\vone_k & \tfrac{1}{k}\vone_k & \ldots & \mI_k
    \end{bmatrix},
\end{align}
where $\vone_k$ denotes the $k \times k$ matrix of all ones. This gives us the conditional squared error 
\begin{align}
    R(\widehat{\vbeta}) = \expect{\phi(y, \vx^\transp \widehat{\vbeta}) \,|\, \widehat{\vbeta}} =  
    \norm[2]{\vbeta^* - \widehat{\vbeta}}^2 
    + \frac{1}{k} \big(\sum_{j=1}^p \beta^*_j - \hat{\beta}_j \big)^2 
    - \frac{1}{k} \sum_{j=1}^d \big( \sum_{j' > (j - 1)k}^{jk} \beta^*_{j'} - \hat{\beta}_{j'} \big)^2 + \frac{1}{2}.
\end{align}
As we show in \Cref{fig:comp-all} (right), randomized ALO is still able to provide an accurate risk estimate even for categorical data. As in the other cases, it provides a more accurate risk estimate in less time than $5$-fold CV, here using the iterative CG solver on sparse $p=20000$-dimensional data.

\subsection{Hyperparameter sweep}

In settings where CV is particularly poorly behaved, RandALO provides a more accurate sense of how risk varies with hyperparameters.
We ran an experiment using the same setup as in \Cref{sec:num:lasso} but with $n = 5000$, $p=25000$, $s=250$, and $\epsilon_i \sim \normal(0, 4)$. Sweeping across a whole range of lasso regularization parameters $\lambda=\lambda_0 / \sqrt{p}$, we show in \Cref{fig:hyperparam-sweep} that RandALO provides an extremely high quality risk estimate with minimal computational overhead, eliminating nearly all of the bias of BKS-ALO and providing a significantly better risk estimate curve over CV with extremely minimal computational overhead.

\begin{figure}[t]
    \centering
    \includegraphics[width=\textwidth]{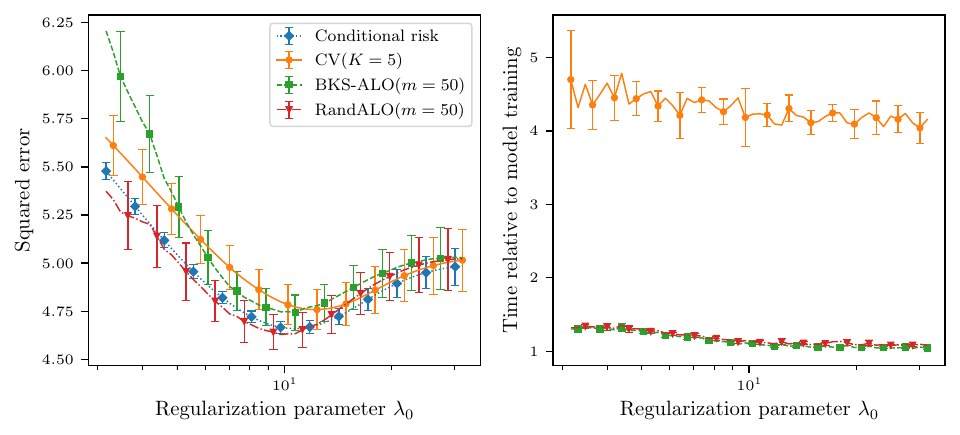}
    \caption{Randomized ALO provides consistent risk estimation across the entire range of regularization parameters, consistently beating $5$-fold CV in both risk estimation and computation. Alarmingly, CV's biased risk curve is minimized by a different value of $\lambda_0$, as demonstrated in \Cref{tab:parameter-selection}. While BKS-ALO is very biased for small values of $\lambda$, this is nearly completely resolved by the debiasing step of RandALO.
    Error bars denote standard deviation over 10 trials.}
    \label{fig:hyperparam-sweep}
\end{figure}
\begin{table}[t]
    \begin{center}
    \begin{tabular}{c|c|ccc|ccc|ccc}
    & &
    \multicolumn{3}{c}{CV($K$)} &
    \multicolumn{3}{c|}{BKS-ALO($m$)} &
    \multicolumn{3}{c|}{RandALO($m$)} \\
    & Conditional risk &
    $2$ & $5$ & $10$ &
    $20$ & $50$ & $100$ &
    $20$ & $50$ & $100$ \\
    \midrule
    $\mathbf{{\vlambda}_0^*=10}$ & \textbf{100} &
    0 & 37 & \textbf{89} & 
    \textbf{67} & \textbf{99} & \textbf{100} &
    \textbf{100} & \textbf{100} & \textbf{100} \\
    $\lambda_0^*=15$ & 0 &
    \textbf{100} & \textbf{63} & 11 &
    33 & 1 & 0 &
    0 & 0 & 0 \\
    \bottomrule
    \end{tabular}
    \end{center}
    \caption{For the same data from \Cref{fig:hyperparam-sweep}, we report which of two hyperparameter values has a lower risk estimate over 100 trials. The bias present in CV, and in BKS-ALO for small $m$, results in the wrong parameter being chosen, particularly often for CV. Meanwhile, the debiased RandALO consistenly selects the correct value of $\lambda_0$ every trial.}
    \label{tab:parameter-selection}
\end{table}

In fact, for this problem, the bias of CV is sufficiently severe as to yield incorrect hyperparameter selection on the basis of the risk estimate. To emphasize this, in \Cref{tab:parameter-selection} we show over the 100 trials how many times the choice $\lambda_0 = 10$, which is around the global minimizer, would be chosen over $\lambda_0 = 15$, which is nearly halfway to null model risk achieved around $\lambda_0 = 30$. Thus for very high-dimensional problems, the risk estimate of RandALO can produce better model selection decisions than CV, in just a fraction of the time.


\subsection{Kernel logistic regression on Fashion-MNIST}

Although we have developed RandALO for linear models, this does not restrict us to models that are linear in the data space. Using \Cref{cor:kernel}, we can also obtain risk estimates for kernel methods with RKHS norm penalties, though with no proven guarantee that ALO provides consistent risk estimation in this setting. In this example on real data from the Fashion-MNIST dataset \citep{xiao2017fashion}, we apply kernel logistic regression using the same loss and regularizer as in \Cref{sec:logistic} on the binary task of differentating ``casual'' (t-shirt, pullover, sneaker) and ``formal'' (shirt, coat, ankle boot) clothing. We select $n=5000$ training samples and 20000 test samples at random and use the radial basis function kernel $e^{-\gamma \norm[2]{\vx - \vx'}^2}$. The data points are $784$-dimensional vectors of pixel intensities taking values from $0$ to $255$.

In \Cref{fig:fashion-mnist}, we compare the resulting risk estimates for 5-fold CV and RandALO as a function of the ridge parameter $\lambda$ and kernel parameter $\gamma$. Both CV and RandALO provide biased risk estimates, and neither selects a parameter that minimizes the test error. Both do select good parameters, with CV achieving $12.40$\% test error and RandALO achieving $11.83$\% test error, compared to the best test error of $11.81$\%, but CV requires nearly triple the amount of computational effort to do so. RandALO provides a better risk estimate than CV where the test error is also small, but outside of the low risk basin the risk estimation is worse, such as for $\gamma = 10^{-5}$, which corresponds to the kernel becoming too narrow for the data and reducing the effectiveness of ALO.

\begin{figure}[t]
    \centering
    \includegraphics[width=\linewidth]{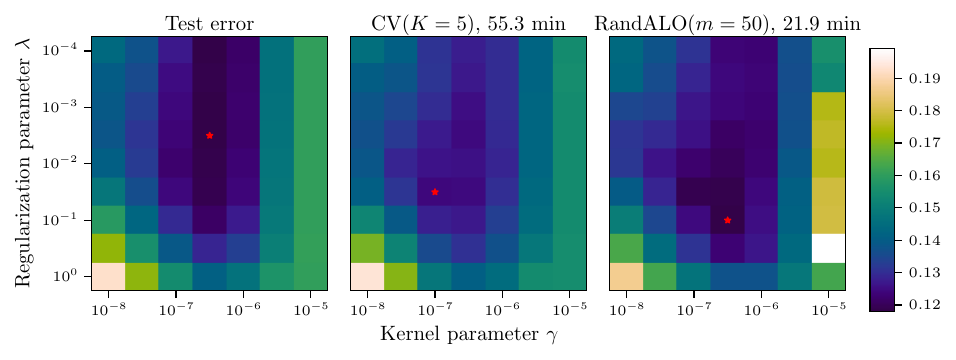}
    \caption{Even for real data from Fashion-MNIST with kernel logistic regression, randomized ALO provides comparable risk estimation to 5-fold CV in almost one third of the computaitonal time, and only taking 3 additional minutes beyond the 18.9 minutes required to train the full model for each $(\lambda, \gamma)$ pair. Both CV and RandALO exhibit some bias with minimizers (red stars) favoring larger values of $\lambda$. For very large $\gamma$, as the kernel matrix approaches $\mI$, the assumptions required for ALO break down and the risk estimate becomes poor.}
    \label{fig:fashion-mnist}
\end{figure}

\subsection{Lasso on genetic data}
In this experiment, we took SNP data from the UK Biobank \citep{ukbiobank} of British men between ages 20 and 60 and trained a model to predict their height.
SNP data takes values in $\{0, 1, 2\}$.
We set aside $10\%$ of the data as a holdout set to validate the risk estimates of RandALO and 5-CV.
We use the Adelie solver \citep{yang2024fast} in its default configuration to fit a lasso model with an unpenalized intercept and an unpenalized first 10 principal components of the SNP matrix.
In this problem, we had $n = 75\,002$, $p = 591\,523$, and used $8\,295$ samples for validation.
As shown in \Cref{fig:genetics-oos-alo}, ALO produces nearly the same prediction error curve as the out-of-sample holdout set.

We report the y-axis as percent of features selected, RandALO's numerical performance becomes poor after $40\%$ of possibly selected features have been selected. 
This is in-part a limitation of the MINRES-based \citep{paige1975solution} solver we used to evaluate \eqref{eq:en-penality-jacobian}.
Further, we do not report runtime performance numbers for this experiment:
Adelie is a coordinate descent method, so computing new entries along the regularization path is relatively cheap.
RandALO is evaluated from scratch for each point on the regularization path.
Future work is needed to design methods to update the Jacobian along the regularization path to perform competitively with Adelie's runtime performance.
Notably, by reusing the selected folds for each 5-CV estimate along the regularization path, 5-CV is able to find its prediction error estimate for the full regularization path faster than RandALO.
However, by reusing the selected folds, the 5-CV estimates for each point on the path are no longer independent conditional on the data.
This risks 5-CV's errors being correlated along the entire regularization path, possibly giving a biased selection of regularizer.

\begin{figure}[t]
    \centering
    \input{figures/genetics.pgf}
    \caption{ On real data from the UK Biobank,  RandALO accurately estimates the risk comaparably or better than CV.
    Notably as the model includes more variables and therefore the Jacobian approaches $\mI$, small errors in the diagonal estimate get amplified by the ${1 - \tilde J_{ii}}$ term in the denominator, causing the risk estimates to become poor.}
    \label{fig:genetics-oos-alo}
\end{figure}

\section{Discussion}

We have presented a randomized method for computing the approximate leave-one-out risk estimate that enables efficient hyperparameter tuning in time comparable to and often much better than $K$-fold cross-validation. The key to our method is combining ALO with randomized diagonal estimation along with the crucial proper handling of estimation noise to reduce bias and variance, resulting in our requiring only a fairly small roughly constant number of quadratic program solves regardless of problem size, scaling very favorably to large datasets.

There are a few important extensions that we leave for subsequent work. Firstly, it is important to extend to non-differentiable losses, which include popular losses such as the hinge loss for support vector machines and the pinball loss for quantile regression. These losses are often also paired with non-differentiable regularizers such as the $\ell_1$ norm, and so it is important to be able to handle joint non-differentiability. With second derivatives of the loss arising very naturally due to the Newton step of ALO, more care is needed in obtaining the appropriate update and making the appropriate adjustments to the randomized diagonal estimation procedure to deal with special non-differentiability behavior.

Secondly, in large-scale machine learning, arguably the most common and important task is multi-class classification. Extending to this case would require first extending ALO to multidimensional outputs and then incorporating an appropriate extension of randomized ``diagonal'' estimation to these higher order tensors. 
With a proper extension to multi-class classification, RandALO could be applied to neural networks taking the Jacobian based perspective of ALO from \Cref{sec:generic-alo}. Based on the results of \citet{park2023trak} who showed compelling results of applying a method based on ALO for data attribution, we anticipate risk estimates based on linearized neural networks performing very well, having the potential to save precious training data from being set aside for validation.




\section*{Acknowledgements}
The authors would like to thank Stephen Boyd, Alice Cortinovis, Pratik Patil, and James Yang for many helpful discussions.
PTN was supported in part by the National Science Foundation Graduate Research Fellowship Program under Grant No.\ DGE-1656518. Any opinions, findings, and conclusions or recommendations expressed in this material are those of the author(s) and do not necessarily reflect the views of the National Science Foundation.
DL was supported by ARO grant 2003514594 and Stanford Data Science. 
EJC was supported by the Office of Naval Research grant N00014-24-1-2305, the National Science Foundation grant DMS-2032014, the Simons Foundation under award 814641.
Some of the computing for this project was performed on the Sherlock cluster. The authors would like to thank Stanford University and the Stanford Research Computing Center for providing computational resources and support that contributed to these research results.

\appendix
\section{Generic derivation of ALO}
\label{sec:generic-alo}

Instead of a linear model $\vx^\transp \vbeta$, consider an arbitrary model $h_\vtheta(\vx)$ differentiably parameterized by $\vtheta \in \reals^q$. The fully trained model $h_{\widehat{\vtheta}}$ satisfies the first-order optimality condition
\begin{align}
    \vzero \in \sum_{i=1}^n \ell'(y_i, h_{\widehat{\vtheta}}(\vx_i)) \nabla_{\widehat{\vtheta}} h_{\widehat{\vtheta}}(\vx_i) + \partial_{\widehat{\vtheta}} r(\widehat{\vtheta}).
    \label{eq:generic-erm-first-order}
\end{align}
We seek to approximate the LOO solution $\widehat{\vtheta}_{-i}$, which satisfies instead
\begin{align}
    \vzero \in \sum_{j \neq i} \ell'(y_j, h_{\widehat{\vtheta}_{-i}}(\vx_j)) \nabla_{\widehat{\vtheta}_{-i}} h_{\widehat{\vtheta}_{-i}}(\vx_j) + \partial_{\widehat{\vtheta}_{-i}} r(\widehat{\vtheta}_{-i}).
    \label{eq:generic-loo-first-order}
\end{align}
The key idea is to start from $\widehat{\vtheta}$ and follow the path of solutions that still satisfy \cref{eq:generic-erm-first-order}; that is, they should still be the solution to a regularized empirical risk minimization problem. As we seek to satisfy \cref{eq:generic-loo-first-order}, this leaves one degree of freedom in \cref{eq:generic-erm-first-order}, namely, the value of $y_i$, which does not appear in \cref{eq:generic-loo-first-order}. Now following \citet{rad2020alo}, we take a Newton step of optimization starting from $\widehat{\vtheta}$ towards a root of the LOO optimality condition. Under \cref{eq:generic-erm-first-order}, this is equivalent to simply finding a root of the left out loss
\begin{align}
    \ell'(y_i, h_\vtheta(\vx_i)) \nabla_\vtheta h_\vtheta(\vx_i).
\end{align}
Assuming that $\nabla_\vtheta h_\vtheta(\vx_i) \neq \vzero$ for any $\vtheta$, a root must be a root of $\ell'(y_i, h_\vtheta(\vx_i))$. Recall that we had only one degree of freedom in $y_i$, and denote $\hat{y}_i = h_{\widehat{\vtheta}}(\vx_i)$, which is a function of $y_i$. Then we can apply one step of Newton's method starting from $\hat{y}_i$ to obtain the ALO prediction:
\begin{align}
    \tilde{y}_i 
    = \hat{y}_i - \frac{\ell'(y_i, \hat{y}_i)}{\frac{d \ell'(y_i, \hat{y}_i)}{d \hat{y}_i}}
    = 
    \hat{y}_i - \frac{\ell'(y_i, \hat{y}_i)}{\frac{\partial \ell'(y_i, \hat{y}_i)}{\partial y_i} \frac{\partial y_i}{\partial \hat{y}_i} + \ell''(y_i, \hat{y}_i)}
    ,
\end{align}
where $d/d \hat{y}_i$ denotes the total derivative and must be taken through both arguments of $\ell'$. Since derivatives of $\ell'$ are inexpensive to evaluate, the fundamental quantity of ALO is the partial derivative $\partial \hat{y}_i/\partial y_i = J_{ii}$, and with reparameterization into $\tilde{J}_{ii}$, we see that this procedure coincides exactly with \cref{eq:alo}.

\section{Proof of \Cref{thm:bks-clt}}
\label{sec:thm:bks-clt:proof}

The concentration result that we will leverage is the Hanson--Wright inequality for $\alpha$-sub-exponential random vectors, generalizing the more common result for sub-Gaussian random vectors. Here the Orlicz (quasi-)norm is defined as 
\begin{align}
    \norm[\psi_\alpha]{X} \defeq \inf \set[\Big]{t > 0 \colon \E \bracket[\Big]{\exp \set[\Big]{\frac{|X|^\alpha}{t^\alpha}}} \leq 2}.
\end{align}
A random variable $X$ is $\alpha$-sub-exponential if $\norm[\psi_\alpha]{X}$ is finite.
\begin{lemma}[Proposition 1.1 and Corollary 1.4, \citealp{gotze2021concentration}]
\label{lem:hanson-wright}
Let $\vx \in \reals^n$ be a random vector with independent components $x_i$ which satisfy $\expect{x_i} = 0$ and $\norm[\psi_\alpha]{x_i} \leq M$. Let $\va \in \reals^n$ and let $\mA$ be a symmetric $n \times n$ matrix. Then for some constant $c_\alpha > 0$ depending only on $\alpha$, for every $t \geq 0$,
\begin{align}
    \Pr(|\vx^\transp \va| \geq t) 
    &\leq 2 \exp \bracket[\Big]{-c_\alpha \min \set[\Big]{\frac{t^2}{M^2 \norm[2]{\va}^2}, \paren[\Big]{\frac{t}{M \norm[\infty]{\va}}}^{\alpha}}} \\
    \Pr(|\vx^\transp \mA \vx - \expect{\vx^\transp \mA \vx}| \geq t) &\leq 2 \exp \bracket[\Big]{-c_\alpha \min \set[\Big]{\frac{t^2}{M^4 \norm[F]{\mA}^2}, \paren[\Big]{\frac{t}{M^2 \norm{\mA}}}^{\frac{\alpha}{2}}}}.
\end{align}
\end{lemma}
From this result we also derive the following corollary for asymmetric quadratic forms.
\begin{corollary}
    \label{cor:asymmetric-hanson-wright}
    Let $\vx \in \reals^m$ and $\vz \in \reals^n$ be random vectors with independent components $x_i, z_i$ which satisfy $\expect{x_i} = \expect{z_i} = 0$ and $\norm[\psi_{\alpha}]{x_i}, \norm[\psi_{\alpha}]{z_i} \leq M$, and let $\mA \in \reals^{m \times n}$ be a matrix. Then for some constant $c_\alpha$ depending only on $\alpha$, for every $t \geq 0$,
    \begin{align}
        \Pr(|\vx^\transp \mA \vz| \geq t) &\leq 2 \exp \bracket[\Big]{-c_\alpha \min \set[\Big]{\frac{t^2}{M^4 \norm[F]{\mA}^2}, \paren[\Big]{\frac{t}{M^2 \norm{\mA}}}^{\frac{\alpha}{2}}}}.
    \end{align}
\end{corollary}
\begin{proof}
First apply \Cref{lem:hanson-wright} for
\begin{align}
    \vx' = \begin{bmatrix}
        \vx \\ \vz
    \end{bmatrix}
    \quad \text{and}
    \quad
    \mA' = \begin{bmatrix}
        \vzero & \mA^\transp \\
        \mA & \vzero
    \end{bmatrix}.
\end{align}
It is straightforward to see that $\norm[F]{\mA'}^2 = 2 \norm[F]{\mA}^2$ and $\norm{\mA'} = \norm{\mA}$, giving us the high probability bound
\begin{align}
    \Pr(|2 \vx^\transp \mA \vz| \geq t) &\leq 2 \exp \bracket[\Big]{-c_\alpha \min \set[\Big]{\frac{t^2}{2 M^4 \norm[F]{\mA}^2}, \paren[\Big]{\frac{t}{M^2 \norm{\mA}}}^{\frac{\alpha}{2}}}}.
\end{align}
By choosing $t' = t/2$, we obtain the stated result as an upper bound.
\end{proof}

We will also use this result on the spectral concentration for random matrices, which will allow us to conclude that certain spectral properties of the random data do not change even if a single data point is left out.
Here the trace norm, or nuclear norm, is defined as $\norm[\tr]{\mTheta} = \tr[(\mTheta \mTheta^\transp)^{1/2}]$.
\begin{lemma}[Theorem 1, \citealp{rubio2011spectral}]
\label{lem:stieltjes-concentration}%
Let $\mZ \in \complexset^{n \times p}$ be a random matrix consisting of i.i.d.\ random variables that have mean 0, variance 1, and finite 
absolute moment of order
$8 + \delta$ for some $\delta > 0$. Let $\mT \in \complexset^{n \times n}$ and $\mSigma \in \complexset^{p \times p}$ be positive semidefinite matrices with operator norm uniformly bounded in $n$, and let $\mX = \mT^{1/2} \mZ \mSigma^{1/2}$.
Then, for $\lambda > 0$,
as $n, p \to \infty$ such that
$0 < \liminf \tfrac{p}{n} \le \limsup \tfrac{p}{n} < \infty$,
we have for any $\mTheta$ having trace norm uniformly bounded in $p$,
\begin{equation}
    \tr \bracket[\Big]{\mTheta \Biginv{\frac{1}{n} \mX^\transp \mX + \lambda \mI}} - \tr \bracket[\Big]{\mTheta \Biginv{\xi \mSigma + \lambda \mI}} \asconv 0,
\end{equation}
where $\xi$ does not depend on $\mZ$ but solves $\xi = \frac{1}{n} \tr [\mT \inv{\mI + \frac{p}{n} \upsilon \mT}] > 0$ and $\upsilon = \frac{1}{p} \tr [\mSigma \inv{\xi \mSigma + \lambda \mI}] > 0$.
\end{lemma}

\begin{proof}[Proof of \Cref{thm:bks-clt}]
Without loss of generality, since $\mG$ can be absorbed into $\mSigma$ for an equivalent problem with $\mG' = n \mI$ and $\mSigma' = n \mG^{-1/2} \mSigma \mG^{-1/2}$, assume $\mG = n \mI$ and $\norm{\mSigma}$ is uniformly bounded. Furthermore, it suffices to prove the result for $m = 1$. Let $\vx_i \in \reals^p$ denote the $i$th row of $\mX$, having the form $\vx_i = \sqrt{t_i} \mSigma^{1/2} \vz_i$, and let $\mX_{-i} \in \reals^{(n-1) \times p}$ denote $\mX$ with $\vx_i$ removed, such that $\vx_i$ is independent of $\mX_{-i}$. Similarly, let $\vw_{-i} \in \reals^{n-1}$ denote $\vw$ with $w_i$ removed.

First, since $w_i$ and $\vw_{-i}$ are independent and zero mean, $\expect{\mu_i | \mX} = \expect{w_i \sum_j \tilde{J}_{ij} w_j | \mX} = \tilde{J}_{ii}$. It remains then to characterize the variance.
By the Woodbury identity, it is straightforward to obtain that
\begin{align}
    \mu_i = \frac{1}{1 + c_i} \paren{c_i + w_i \vx_i^\transp \inv{\mX_{-i}^\transp \mX_{-i} + n \mI} \mX_{-i}^\transp \vw_{-i}}
    \quad\text{where}\quad
    c_i = \vx_i^\transp \inv{\mX_{-i}^\transp \mX_{-i} + n \mI} \vx_i.
\end{align}
The first term, $c_i / (1 + c_i)$, is simply $\tilde{J}_{ii}$.
To determine the value of $c_i$, we can first apply \Cref{lem:hanson-wright} with $\vx = \vz_i$ and $\mA = t_i \mSigma^{1/2} \inv{\mX_{-i}^\transp \mX_{-i} + n \mI} \mSigma^{1/2}$. Note that both $\norm[F]{\mA}$ and $\norm{\mA}$ can be uniformly upper bounded by $C / n$ for some constant $C$. Therefore, 
\begin{align}
    \Pr(|c_i - t_i \tr[\mSigma \inv{\mX_{-i}^\transp \mX_{-i} + n \mI}]| > t) \leq 2 \exp \bracket[\Big]{-c_\alpha \min \set[\Big]{\frac{n^2 t^2}{C^2 M^4}, \paren[\Big]{\frac{nt}{C M^2}}^{\frac{\alpha}{2}}}}.
\end{align}
For every $t > 0$, if we sum the right-hand side over $n = 1$ to $\infty$, the sum of probabilities is finite. Thus by the Borel--Cantelli lemma, we have $c_i - t_i \tr[\mSigma \inv{\mX_{-i}^\transp \mX_{-i} + n \mI}] \asconv 0$. Now applying \Cref{lem:stieltjes-concentration} with $\mTheta = \mSigma / n$, we see that the trace term takes asymptotically almost surely the same value regardless of whether $\vx_i$ is left out of $\mX$ or not, since $p/n - p/(n - 1) \to 0$. Thus $c_i - t_i \eta \asconv 0$ for $\eta \defeq \tr [\mSigma \inv{\mX^\transp \mX + n \mI}]$.

Next, we need to apply a central limit theorem to obtain Gaussianity of the error. Note that
\begin{align}
    w_i \vx_i^\transp \inv{\mX_{-i}^\transp \mX_{-i} + n \mI} \mX_{-i}^\transp \vw_{-i}
    = \sum_{j \neq i} w_i \vx_i^\transp \inv{\mX_{-i}^\transp \mX_{-i} + n \mI} \vx_j w_j.
\end{align}
Since $w_j$ are independent of the remaining quantities, we can apply Lyapunov's central limit theorem provided we can show that the other terms are not too sparse. That is, we need to show that for some $\delta > 0$,
\begin{align}
    \frac{\sum_{j \neq i} |\vx_i^\transp \inv{\mX_{-i}^\transp \mX_{-i} + n \mI} \vx_j|^{2 + \delta}}{(\vx_i^\transp \inv{\mX_{-i}^\transp \mX_{-i} + n \mI} \mX_{-i}^\transp \mX_{-i} \inv{\mX_{-i}^\transp \mX_{-i} + n \mI} \vx_i)^{\frac{2 + \delta}{2}}} \asconv 0.
    \label{eq:lyapunov-limit}
\end{align}
To deal with the numerator, note that by the Woodbury identity,
\begin{align}
    \vx_i^\transp \inv{\mX_{-i}^\transp \mX_{-i} + n \mI} \vx_j = \frac{\vx_i^\transp \inv{\mX_{-ij}^\transp \mX_{-ij} + n \mI} \vx_j}{1 + \vx_j^\transp \inv{\mX_{-ij}^\transp \mX_{-ij} + n \mI} \vx_j},
    \label{eq:quadratic-ij-woodbury}
\end{align}
where $\mX_{-ij}$ is the matrix $\mX$ with both the $i$th and $j$th rows removed. We can then apply \Cref{cor:asymmetric-hanson-wright} with $\vx = \vz_i$, $\vz = \vz_j$, and $\mA = \mSigma^{1/2} \inv{\mX_{-ij}^\transp \mX_{-ij} + n \mI} \mSigma^{1/2}$, which must have $\norm[F]{\mA}$ and $\norm{\mA}$ upper bounded by $C / n$. Choosing $t = 1 / \sqrt{n}$, we have
\begin{align}
    \Pr\paren[\Big]{|\vx_i^\transp \inv{\mX_{-ij}^\transp \mX_{-ij} + n \mI} \vx_j| \geq \frac{1}{\sqrt{n}}} &\leq 2 \exp \bracket[\Big]{-c_\alpha \min \set[\Big]{\frac{n}{M^4 C^2}, \paren[\Big]{\frac{\sqrt{n}}{M^2 C}}^{\frac{\alpha}{2}}}}.
    \label{eq:x_ij_vanish}
\end{align}
Using a union bound over $j \neq i$, we can therefore say that all of these terms are bounded by $1 / \sqrt{n}$ with probability at least $1 - 2 \exp[-c n^{\alpha / 4} + \log n]$. With similar probability, the denominator of \cref{eq:quadratic-ij-woodbury} also concentrates for all $j \neq i$. Thus, we have
\begin{align}
    \sum_{j \neq i} |\vx_i^\transp \inv{\mX_{-i}^\transp \mX_{-i} + n \mI} \vx_j|^{2 + \delta}
    \overset{\mathrm{w.h.p.}}{\leq}
    \frac{n - 1}{n^{\frac{2 + \delta}{2}}}
    \asconv 0,
\end{align}
where ``$\mathrm{w.h.p.}$'' indicates that the inequality holds with sufficiently high probability that the sum of the complementary events across $n$ are finite, such that we can apply the Borel--Cantelli lemma to obtain almost sure convergence.
We now need only show that the denominator of \cref{eq:lyapunov-limit} is lower bounded. Making similar concentration arguments as above using \Cref{lem:hanson-wright}, it is sufficient to show that
\begin{align}
    \tr[\mSigma \inv{\mX_{-i}^\transp \mX_{-i} + n \mI} \mX_{-i}^\transp \mX_{-i} \inv{\mX_{-i}^\transp \mX_{-i} + n \mI}]
    \label{eq:nu-expression}
\end{align}
is uniformly bounded away from 0. Since $\mSigma$ has all eigenvalues bounded away from zero, we can ignore it, and so we need only lower bound the smallest singular value of $\mX_{-i}$. Again since $\mT$ and $\mSigma$ have lower bounded eigenvalues, we need only the smallest singular value of $\mZ_{-i}$ to be lower bounded by $c \sqrt{n}$, which we have by classical results in random matrix theory \citep{silverstein1998eigenvalues} almost surely. Then $\tr[\mX_{-i}^\transp \mX_{-i} \inv[2]{\mX_{-i}^\transp \mX_{-i} + n \mI}] \geq cp/(c + 1)^2n > 0$, and so we can apply Lyapunov's central limit theorem to obtain that almost surely over $\mX$,
\begin{align}
    \frac{w_i \vx_i^\transp \mSigma^{1/2} \inv{\mX_{-i}^\transp \mX_{-i} + n \mI} \mX_{-i}^\transp \vw_{-i}}
    {\sqrt{\vx_i^\transp \inv{\mX_{-i}^\transp \mX_{-i} + n \mI} \mX_{-i}^\transp \mX_{-i} \inv{\mX_{-i}^\transp \mX_{-i} + n \mI} \vx_i}}
    \dconv \normal(0, 1).
\end{align}
To get a simpler value for the expression in \cref{eq:nu-expression}, recall that the derivative of a matrix resolvent is a second order resolvent: $\partial/\partial \lambda \inv{\mA + \lambda \mI} = -\inv[2]{\mA + \lambda \mI}$, meaning that we can apply \Cref{lem:stieltjes-concentration} to second-order resolvent polynomials as well (such ``asymptotic equivalences'' hold for derivatives provided all sequences are bounded as they are in our case; see, Theorem~11 of \citealp{dobriban2020wonder}), allowing us to replace $\mX_{-i}$ with $\mX$:
\begin{multline}
    \tr\bracket[\big]{\mSigma \inv{\mX_{-i}^\transp \mX_{-i} + n \mI} \mX_{-i}^\transp \mX_{-i} \inv{\mX_{-i}^\transp \mX_{-i} + n \mI}} \\
    - \underbrace{\tr \bracket[\big]{\mSigma \inv{\mX^\transp \mX + n \mI} \mX^\transp \mX \inv{\mX^\transp \mX + n \mI}}}_{\rdefeq \nu} \asconv 0.
\end{multline}

It remains to show that the element-wise noise of $\mu$ is asymptotically conditionally uncorrelated. That is, we need to show that for $i \neq j$.
\begin{align}
    \E \bracket[\Big]{w_i \vx_i^\transp \inv{\mX^\transp \mX + n \mI} \mX_{-i}^\transp \vw_{-i}
    w_j \vx_j^\transp \inv{\mX^\transp \mX + n \mI} \mX_{-j}^\transp \vw_{-j} \Big| \mX} \to 0.
\end{align}
Note that this term can be expressed as a sum $\sum_{k\ell}w_i A_{ik} w_k w_j B_{j\ell} w_\ell$, and that we can therefore exploit the fact that $\expect{w_i} = \expect{w_i^3} = 0$. Observe that we can never have $i = k$ or $j = \ell$, and of course $i \neq j$. The remaining term, where $k = j$ and $\ell = i$, vanishes:
\begin{align}
    \vx_i^\transp \inv{\mX^\transp \mX + n \mI} \vx_j
    \vx_j^\transp \inv{\mX^\transp \mX + n \mI} \vx_i \overset{\mathrm{w.h.p.}}{\leq} \frac{1}{n} \asconv 0,
\end{align}
which follows by a similar arguments based on \Cref{cor:asymmetric-hanson-wright} made earlier in the proof (first apply the Woodbury identity twice to extract $\vx_i$ and $\vx_j$ from the inverse as a shrinking scalar, then apply \cref{eq:x_ij_vanish}), proving the stated result for $m = 1$.
To account for $m > 1$, we can simply average Gaussian variables.
\end{proof}

\bibliographystyle{abbrvnat}
\bibliography{references}

\end{document}